\documentclass[12pt]{article}
\usepackage{amsmath}
\usepackage{latexsym}
\usepackage{amssymb}
\newtheorem{thm}{Theorem}[section]
\newtheorem{la}[thm]{Lemma}
\newtheorem{Defn}[thm]{Definition}
\newtheorem{Exa}[thm]{Example}
\newtheorem{Remark}[thm]{Remark}
\newtheorem{prop}[thm]{Proposition}
\newtheorem{cor}[thm]{Corollary}
\newtheorem{Number}[thm]{\!\!}
\newenvironment{defn}{\begin{Defn}\rm}{\end{Defn}}
\newenvironment{exa}{\begin{Exa}\rm}{\end{Exa}}
\newenvironment{rem}{\begin{Remark}\rm}{\end{Remark}}
\newenvironment{numba}{\begin{Number}\rm}{\end{Number}}
\newenvironment{proof}{{\noindent\bf Proof.}}%
                  {\nopagebreak\hspace*{\fill}$\Box$\medskip\par}
\newcommand{\wb}{\overline}
\newcommand{\ve}{\varepsilon}
\newcommand{\wt}{\widetilde}
\newcommand{\mto}{\mapsto}
\newcommand{\N}{{\mathbb N}}
\newcommand{\R}{{\mathbb R}}
\newcommand{\C}{{\mathbb C}}
\newcommand{\K}{{\mathbb K}}
\newcommand{\bL}{{\mathbb L}}
\newcommand{\bS}{{\mathbb S}}
\newcommand{\cg}{{\mathfrak g}}

\newcommand{\cL}{{\mathcal L}}
\newcommand{\cO}{{\mathcal O}}
\newcommand{\cT}{{\mathcal T}}
\newcommand{\cU}{{\mathcal U}}
\newcommand{\cF}{{\mathcal F}}
\newcommand{\cS}{{\mathcal S}}

\newcommand{\cA}{{\mathcal A}}
\newcommand{\cE}{{\mathcal E}}
\newcommand{\cG}{{\mathcal G}}

\newcommand{\sub}{\subseteq}
\newcommand{\wh}{\widehat}
\DeclareMathOperator{\id}{id}
\DeclareMathOperator{\Supp}{supp}
\DeclareMathOperator{\graph}{graph}
\DeclareMathOperator{\locc}{loc}
\DeclareMathOperator{\rc}{rc}
\newcommand{\loc}{{\locc}}
\newcommand{\dl}{{\displaystyle \lim_{\longrightarrow}\, }}
\newcommand{\defi}{:=}
\DeclareMathOperator{\Fl}{Fl}
\DeclareMathOperator{\im}{im}
\DeclareMathOperator{\glue}{glue}
\DeclareMathOperator{\pr}{pr}
\DeclareMathOperator{\ev}{ev}
\DeclareMathOperator{\Ad}{Ad}
\begin{document}
\begin{center}
{\Large\bf
Mapping groups associated with\vspace{1.4mm} real-valued function spaces and
direct limits\vspace{1.4mm} of Sobolev-Lie groups}\\[7mm]
{\bf Helge Gl\"{o}ckner and Luis T\'{a}rrega}\vspace{3mm}
\end{center}
\begin{abstract}
\hspace*{-5.5mm}Let $M$ be a compact smooth manifold
of dimension~$m$ (without boundary)
and $G$ be a finite-dimensional Lie group, with Lie algebra~$\cg$.
Let $H^{>\frac{m}{2}}(M,G)$ be the group of all mappings
$\gamma\colon M\to G$ which are $H^s$ for some $s>\frac{m}{2}$.
We show that $H^{>\frac{m}{2}}(M,G)$ can be made
a regular Lie group in Milnor's sense,
modelled on the Silva space
$H^{>\frac{m}{2}}(M,\cg):={\dl}_{s>\frac{m}{2}}H^s(M,\cg)$,
such that
\[
H^{>\frac{m}{2}}(M,G)\; =\;\, {\dl}_{s>\frac{m}{2}}H^s(M,G)
\]
as a Lie group (where $H^s(M,G)$ is the Hilbert-Lie group
of all $G$-valued $H^s$-mappings on~$M$).
We also explain how the (known)
Lie group structure on $H^s(M,G)$
can be obtained as a special case of
a general construction of Lie groups $\cF(M,G)$,
whenever function spaces $\cF(U,\R)$
on open subsets $U\sub\R^m$ are given,
subject to simple axioms.
\end{abstract}
{\bf Classification:}
22E65 (primary);
22E67, % loop groups etc
46A13, % LB LF etc
46E35, % Sobolev
46M40 % ind and proj lim in fa
\\[2.3mm]
\section{Introduction and statement of results}
Lie groups of mappings
from a compact manifold~$M$
to a finite-dimensional Lie group~$G$
form an important class of infinite-dimensional
Lie groups,
as well as variants like gauge groups
of principal $G$-bundles over~$M$.
See \cite{Nee} for more context, as well as
the references at the end of this introduction.\\[2.3mm]
In this article, we describe a general construction principle
for Lie groups of mappings when
real-valued function spaces are given,
satisfying suitable axioms.
We then study ascending unions
of the constructed mapping groups,
in the special case of Sobolev-Lie groups.\\[2.3mm]
For fixed $m\in\N$, consider
a basis $\cU$ of the topology
of~$\R^m$ satisfying suitable
properties (a ``good collection of open sets''
in the sense of Definition~\ref{good-coll}).
Suppose that, for each $U\in\cU$,
an integral complete locally convex space $\cF(U,\R)$
of bounded, continuous real-valued
functions is given.\footnote{We say that
a locally convex space~$F$ is \emph{integral complete}
if the weak integral $\int_0^1 \gamma(t)\,dt$
exists in~$F$ for each continuous map
$\gamma\colon [0,1]\to F$.
See~\cite{Wei}
for a characterization.}
Then an integral complete
locally convex space $\cF(U,E)$
of $E$-valued maps
can be defined in a natural way
for each finite-dimensional real
vector space~$E$ (see~\ref{vector_espace_E}).
If four simple axioms
(PF), (PB), (GL), and (MU)
are satisfied, we say that the family
$(\cF(U,\R))_{U\in\cU}$
is \emph{suitable for Lie theory} (see Definition~\ref{theaxioms}).
For each $E$ as before,
one can then define
a locally convex space $\cF(M,E)$
of $E$-valued functions for each compact
$m$-dimensional smooth manifold~$M$
without boundary, see \ref{the-ini-top}.
We can also define a set
$\cF(M,N)$ of
$N$-valued functions on~$M$,
for each finite-dimensional smooth manifold~$N$
(see~\ref{f-to-N}).
If~$N$ is a Lie group,
we obtain (with terminology as in \ref{defn-bch}):
\begin{prop}\label{fi-prop}
Let $\cU$ be a good collection of open subsets
of~$\R^m$ and
$(\cF(U,\R))_{U\in\cU}$ be a
family of integral complete
locally convex spaces which is suitable
for Lie theory.
Let $M$ be a compact $m$-dimensional smooth manifold
without boundary and $G$ be a finite-dimensional Lie group
over $\K\in\{\R,\C\}$, with Lie algebra~$\cg$.
Then $\cF(M,G)$ can be made a $\K$-analytic
BCH-Lie group whose Lie algebra can be identified
with $\cF(M,\cg)$, such that
\[
\cF(M,\exp_G)\colon \cF(M,\cg)\to \cF(M,G),\quad
\gamma\mto \exp_G\circ\, \gamma
\]
is the exponential function of $\cF(M,G)$.
\end{prop}
\begin{exa}
For $m\in\N$ and $s>\frac{m}{2}$,
we can apply Proposition~\ref{fi-prop}
to the Sobolev spaces $\cF(U,\R):=H^s(U,\R)$
on bounded open subsets $U\sub\R^m$
(see Section~\ref{sobolev-suitable}).
We obtain Hilbert-Lie groups $H^s(M,G):=\cF(M,G)$
with properties as described in
the proposition.
\end{exa}
Also the following two examples
can be treated (see \cite{FUR}).
\begin{exa}
For $m\in\N$, $k\in\N_0$, and $\alpha\in\;]0,1]$, 
Proposition~\ref{fi-prop}
can be applied to the Banach spaces $\cF(U,\R):=C^{k,\alpha}(U,\R)$
of $k$ times H\"{o}lder-differentiable functions
on bounded open
subsets $U\sub\R^m$ (see \cite{FUR}).
This yields Banach-Lie groups $C^{k,\alpha}(M,G):=\cF(M,G)$
with properties as described in
the proposition.
\end{exa}
\begin{exa}
For $m=1$ and $p\in [1,\infty]$,
Proposition~\ref{fi-prop}
can be applied to the Banach spaces $\cF(U,\R):=AC_{L^p}(U,\R)$
of absolutely continuous functions
with $L^p$-derivatives
on bounded open intervals $U\not=\emptyset$
in~$\R$
(see \cite{FUR}).
For $M:=\bS$ the unit circle,
this yields Banach-Lie groups $AC_{L^p}(\bS,G):=\cF(\bS,G)$
with properties as described in
the proposition.
\end{exa}
We then study direct limits of
the Hilbert-Lie groups $H^s(M,G)$
as $s\searrow s_0$ for some $s_0\geq m/2$.
Using terminology
as in \ref{defn-bch} and \ref{defn-reg-gp},
we obtain:
\begin{thm}\label{dirlim-1}
Let $m\in\N$, $s_0\geq \frac{m}{2}$
be a real number,
$M$ be a compact, $m$-dimensional smooth
manifold without boundary
and $G$ be a finite-dimensional Lie group
over $\K\in\{\R,\C\}$, with Lie algebra~$\cg$.
Then
\[
H^{>s_0}(M,G):=\bigcup_{s>s_0}H^s(M,G)
\]
can be made a $\K$-analytic BCH-Lie group over~$\K$
whose Lie algebra can be identified
with the locally convex direct limit
\[
H^{>s_0}(M,\cg):={\dl}_{s>s_0}\, H^s(M,\cg),
\]
such that $H^{>s_0}(M,\exp_G)\colon H^{>s_0}(M,\cg)\to H^{>s_0}(M,G)$,
$\gamma\mto \exp_G\circ\, \gamma$
is the exponential function of $H^{>s_0}(M,G)$.
The Lie group $H^{>s_0}(M,G)$ is $L^\infty_{\rc}$-regular and
$C^0$-regular.
%It admits a special direct limit
%chart and
Each compact subset
of $H^{>s_0}(M,G)$
is a compact subset of $H^s(M,G)$ for some
$s>s_0$. Moreover,
\[
H^{>s_0}(M,G)={\dl}_{s>s_0}\, H^s(M,G)
\]
holds in each of the categories
of topological spaces, topological groups,
$C^\infty_\bL$-Lie groups for $\bL\in\{\R,\K\}$,
and $C^r_\bL$-manifolds for $r\in \N_0\cup\{\infty\}$.
\end{thm}
The Lie groups and manifolds we are referring to are
Lie groups and manifolds modelled on locally
convex spaces. The morphisms in the categories
just mentioned are continuous maps,
continuous group homomorphisms,
group homomorphisms which are $C^\infty_\bL$-maps,
and $C^r_\bL$-maps, respectively.\\[3mm]
{\bf General background of the studies.}
Paradigmatic examples of mapping groups
are Lie groups $C^k(M,G)$ of $C^k$-maps
for $k\in\N_0\cup\{\infty\}$,
in particular for $k=\infty$
(see \cite{GCX,KaM,Mck,Mil,PaS}).
Lie groups $H^s(M,G)$
of Sobolev maps with real exponent
$s>m/2$ have also been considered.
See \cite{HaM} for the case of loop groups
(i.e.,
when $M$ is the unit circle $\bS\sub\C$),
using Fourier series for the definition.
For $G$ a compact Lie group
and $(M,g)$ a compact Riemannian manifold,
Lie groups $H^s(M,G)$
are constructed in \cite[Theorem 1.2]{Fre},
referring to \cite[Appendix A]{FaU} for details
where some proofs rely on integer exponents.
A global approach using the Laplace
operator of $(M,g)$
is used there to define Sobolev
spaces.
For real $s>m/2$ and a finite-dimensional
Lie group~$G$,
Sobolev-Lie groups $H^s(M,G)$
also occur in \cite[(7), p.\,395]{Pic}.\\[2.3mm]
Related studies of manifold structures
on $C^k(M,N)$
for a finite-dimensional
smooth manifold~$N$
can be found, e.g., in
\cite{Eel,Ham,KaM}.
Manifold structures on $H^s(M,N)$
for integers $s>\dim(M)/2$
are studied in \cite[p.\,781]{Eel}
and~\cite{IKT};
the possible generalization to real~$s$
is broached in \cite[Appendix~B]{IKT}.\\[2.3mm]
We mention that \cite{Pal}
pursues an axiomatic approach
to global analysis, starting with the choice
of a Banach space-valued section functor
(see \cite[\S4]{Pal}).\footnote{General Sobolev spaces occur
in \cite[\S9]{Pal},
but an essential proof (of \cite[Lemma~9.9]{Pal})
presumes integer exponents.} Starting with function spaces on open subsets
of $\R^m$, as proposed in this article,
constitutes a complementary, more elementary approach.\\[3mm]
{\bf Acknowledgements.}
The first author thanks Rafael Dahmen (now KIT Karlsruhe)
for discussions in the early stages of the project.
The second author acknowledges the support of Universitat Jaume I (P1-1B2015-77 project and E-2016-37 grant).
\section{Preliminaries}
We write $\N=\{1,2,\ldots\}$
and $\N_0=\N\cup\{0\}$.
A map between topological spaces shall be called
a \emph{topological embedding} if it is a homeomorphism
onto its image.
The word ``vector space''
refers to a real vector space,
unless the contrary is stated.
A subset $U$ of a $\K$-vector space~$E$
over $\K\in\{\R,\C\}$
is called \emph{balanced}
if $zx\in U$
for all $x\in U$
and $z\in\K$ with $|z|\leq 1$.
All locally convex
topological vector spaces
are assumed Hausdorff.
If $(E,\|\cdot\|)$
is a normed space,
we let $B^E_r(x):=\{y\in E\colon \|y-x\|<r\}$
be the open ball
of radius $r>0$ around $x\in E$.
We shall use $C^k$-maps
between open subsets
of locally convex spaces
as introduced by Bastiani~\cite{Bas},
and recall some concepts
for the reader's convenience.
For further information, see \cite{RES}
and \cite{GaN} (where also the corresponding manifolds
and Lie groups are discussed),
or also \cite{Ham} (for Fr\'{e}chet
modelling spaces) and \cite{Mil}
(for sequentially complete
spaces). If $U\sub \R^m$
is open and $E$ a finite-dimensional vector space,
we let $C^\infty_c(U,E)$
be the vector space of all compactly
supported smooth functions $\gamma\colon U\to E$.
\begin{numba}
Let $E$ and $F$ be locally convex spaces over
$\K\in\{\R,\C\}$, and $U\sub E$ be open.
A mapping $f\colon U\to F$
is called $C^0_\K$ if it is continuous.
We call $f$ a $C^1_\K$-map if~$f$
is continuous, the directional
derivative
\[
df(x,y):=\lim_{z\to 0}\frac{1}{z}(f(x+zy)-f(x))
\]
exists in~$E$ for all $x\in U$ and $y\in E$
(where $z\in \K\setminus\{0\}$
with $x+zy\in U$), and $df\colon U\times E\to F$
is continuous.
Recursively, for $k\in\N$
we say that $f$ is $C^{k+1}_\K$
if $f$ is $C^1_\K$ and $df$ is $C^k_\K$.
If $f$ is $C^k_\K$
for all $k\in\N$,
then $f$ is called~$C^\infty_\K$.
\end{numba}
The $C^k_\R$-maps are also referred to as $C^k$-maps.
The $C^\infty_\R$-maps are also called
\emph{smooth}.
The $C^\infty_\C$-maps
are also called \emph{complex analytic} (or $\C$-analytic);
they are continuous and given locally
by pointwise convergent series of continuous
complex homogeneous polynomials
(see \cite[Corollary 2.1.9]{GaN},
cf.\ also Proposition 5.5 and Theorem 3.1 in \cite{BaS}), but we shall not use this fact.
For each $C^1_\K$-map $f\colon U\to F$ and each $x\in U$,
%
% koennte hier zwei Zeilen sparen durch Umformulieren:
%
the map
\[
f'(x)\colon E\to F,\quad y\mto df(x,y)
\]
is $\K$-linear.
\begin{numba}
Let $E$ and $F$ be real locally convex spaces
and $U\sub E$ be an open subset.
A function $f\colon U\to F$ is called \emph{real
analytic} (or $\R$-analytic)
if $f$ admits a complex analytic
extension $g\colon W\to F_\C$
to an open subset $W\sub E_\C$
(see \cite[Definition 2.2.2]{GaN},
also \cite{RES} and \cite{Mil}).
\end{numba}
The following fact is useful.
\begin{numba}\label{real-to-c}
Let $E$ and $F$ be complex locally
convex spaces and $U\sub E$
be an open subset. A function
$f\colon U\to F$ is complex analytic
if and only if $f$ is $C^\infty_\R$
and
$f'(x)\colon E\to F$ is $\C$-linear
for each $x\in U$ (see \cite[Corollary 2.1.9]{GaN}, also \cite{RES}).
\end{numba}
\begin{numba}
$C^r_\K$-manifolds modeled on a locally convex topological $\K$-vector space~$E$
for $r\in \N_0\cup\{\infty\}$
and $\K$-analytic manifolds modelled on~$E$
can be defined as expected,
as well as tangent bundles and tangent maps;
likewise $C^\infty_\K$-Lie groups
modelled on~$E$
and $\K$-analytic Lie groups
(see \cite{GaN} and
\cite{Nee}, also \cite{BGN} and \cite{RES}).
When we speak about Lie groups or manifolds,
they may always have infinite dimension
(unless the contrary is stated).
\end{numba}
By definition, a \emph{$\K$-analytic
diffeomorphism} is an
invertible $\K$-analytic
map between $\K$-analytic manifolds
whose inverse is $\K$-analytic.
If $V$ is an open subset of a locally
convex space~$E$, we identify its tangent bundle with
$V\times E$ as usual.
If $M$ is a $C^1$-manifold and $f\colon M\to V$ a $C^1$-map,
we write $df$ for the second component of
the tangent map $Tf\colon TM\to TV=V\times E$.\\[2.3mm]
See \cite{GaN}, \cite{Mil}, and \cite{Nee}
for basic concepts concerning infinite-dimensional
Lie groups (like the Lie algebra $\cg:=L(G):=T_eG$,
the Lie algebra homomorphism
$L(f):=T_e(f)$ associated with a smooth
group homomorphism~$f$
and the notion of an exponential function
$\exp_G\colon \cg\to G$).
See \cite{GaN} for the
%following
next concept
(cf.\ also \cite{GCX,Mil}).
\begin{numba}\label{defn-bch}
A $\K$-analytic Lie group~$G$
is called a \emph{BCH-Lie group}
if it has an exponential function
$\exp_G$
which restricts
to a $\K$-analytic diffeomorphism
from an open zero-neighbourhood
in the Lie algebra~$\cg$ of~$G$
onto an open identity-neighbourhood in~$G$.
\end{numba}
\begin{defn}\label{good-coll}
Let $m\in\N$.
A set $\cU$ of open
%non-empty
subsets
of $\R^m$
will be called a
\emph{good collection of open subsets}
if the following conditions are satisfied:
\begin{itemize}
\item[(a)]
$\cU$ is a basis for the topology of $\R^m$.
\item[(b)]
If $U\in \cU$ and $K\sub U$
is a compact non-empty subset,
then there exists $V\in \cU$
with compact closure $\wb{V}$ in~$\R^m$
such that $K\sub V$ and $\wb{V}\sub U$.
\item[(c)]
If $U\sub\R^m$ is an open set
and $W\in\cU$ is a relatively
compact subset of~$U$,
then there exists $V\in \cU$
such that $V$ is a relatively compact subset of~$U$
and $\wb{W}\sub V$.
\item[(d)]
If $\phi\colon U\to V$
is a $C^\infty$-diffeomorphism
between open subsets $U$ and $V$ of $\R^m$
and $W\in\cU$ is a relatively compact subset of~$U$,
then $\phi(W)\in \cU$.
\end{itemize}
\end{defn}
\begin{exa}\label{exa1}
The following are good
collections
of open subsets of~$\R^m$:
\begin{itemize}
\item[(a)]
The set of all open subsets
of $\R^m$, and the set of all open bounded
subsets;
\item[(b)]
If $m=1$, the set of all relatively compact,
open intervals $I\not=\emptyset$ in~$\R$.
\end{itemize}
\end{exa}
The simple verification is left to the reader.
\begin{rem}\label{rem-smoobd}
In Appendix~\ref{appA}, we show that
also bounded open subsets $U\sub\R^m$
with $C^\infty$-boundary
form a good collection of open sets.
We shall not use this fact here;
but it might be useful for
more complicated potential examples,
like $L^p$-Sobolev spaces for $p\not=2$.
\end{rem}
\section{Axioms for function spaces}\label{sec-genF}
Fix $m\in\N$.
If $U\sub\R^m$
is an open subset, we let $BC(U,\R)$
be the vector space of all bounded continuous
functions $f\colon U\to \R$
and make it a Banach space using the supremum norm~$\|\cdot\|_\infty$.
Let $\cU$ be a good collection of open subsets
of~$\R^m$. For $U\in\cU$,
let a vector subspace $\cF(U,\R)$ of
$BC(U,\R)$ be given;
assume that $\cF(U,\R)$ is equipped with an
integral complete locally convex vector topology making the inclusion $\cF(U,\R)\rightarrow BC(U,\R)$ continuous.\\[2.3mm]
Given $U\in\cU$,
we can then associate an integral complete
locally convex space $\cF(U,E)$
to each finite-dimensional real vector space~$E$:
\begin{numba}\label{vector_espace_E}\label{topo-funs}
If $b_1,\ldots,b_n$ is a basis for~$E$, where $n:=\dim(E)$, we define
\[
\cF(U,E) := \sum^n_{k=1}\cF(U,\R)b_k
\]
and give it the locally vector topology making the map
\begin{equation}\label{an-iso}
\cF(U,\R)^n\rightarrow \cF(U,E),\quad
(f_1,\ldots,f_n)\mapsto \sum^m_{k=1}f_kb_k
\end{equation}
an isomorphism of topological vector spaces.
\end{numba}
Note that
$\cF(U,E)$ and its topology are independent of the choice of basis.
\begin{numba}\label{easiest-prod}
If $E=E_1\oplus E_2$ with vector subspaces $E_1$ and $E_2$,
we can choose a basis $b_1,\ldots,b_k$
for $E_1$ and a basis $b_{k+1},\ldots, b_n$ for $E_2$.
We easily deduce that $\cF(U,E)=\cF(U,E_1)\oplus \cF(U,E_2)$
as a topological vector space.
For all finite-dimensional
vector spaces $F_1$ and $F_2$, we therefore have
\[
\cF(U,F_1\times F_2)\cong \cF(U,F_1)\times \cF(U,F_2).
\]
\end{numba}
If~$W$ is an open subset of~$E$, we
let $\cF(U,W)$ be the set of all
$\gamma\in\cF(U,E)$ such that $\gamma(U)+Q\subseteq W$
for some $0$-neighbourhood $Q\sub E$.
\begin{defn}\label{theaxioms}
We say that $(\cF(U,\R))_{U\in\cU}$ as before
is a \emph{family of locally convex spaces
suitable for Lie theory}
if the following 
axioms are satisfied for all finite-dimensional real vector spaces $E$ and~$F$:
\begin{description}
\item[Pushforward Axiom (PF):]
For all $U,V\in \cU$ such that $V$ is relatively compact in~$U$
and each smooth map $f\colon U\times E\to F$, we have
$f_{\ast}(\gamma) := f\circ (\id_V,\gamma|_V)\in \cF(V,F)$ for all
$\gamma\in \cF(U,E)$ and the map
\[
f_{\ast}\colon \cF(U,E)\to \cF(V,F),\;\;
\gamma\mapsto f_{\ast}(\gamma)
\]
is continuous.
\item[Pullback Axiom (PB):]
Let $U$ be an open subset of~$\R^m$
and $V, W\in\cU$ such that~$W$ has a compact closure
contained in~$U$.
Let $\Theta\colon U\to V$ be a $C^\infty$-diffeomorphism.
Then $\gamma\circ \Theta|_W\in \cF(W,E)$ for all
$\gamma\in \cF(V,E)$ and $\cF(\Theta|_W,E)\colon \cF(V,E)\to \cF(W,E)$,
$\gamma\mapsto \gamma\circ\Theta|_W$
is a continuous map.
\item[Globalization Axiom (GL):]
If $U, V\in \cU$
with
$V\sub U$
and $\gamma\in \cF(V,E)$ has compact support,
then the map $\wt{\gamma}\colon U\rightarrow E$ defined by
$\wt{\gamma}(x)=\gamma(x)$ if $x\in V$ and
$\wt{\gamma}(x)=0$ if $x\in U\setminus \Supp(\gamma)$ is in
$\cF(U,E)$ and for each compact subset~$K$ of~$V$ the map
\[
e^E_{U,V,K}\colon \cF_K(V,E)\to \cF(U,E), \quad
\gamma\mapsto\wt{\gamma}
\]
is continuous, where
$\cF_K(V,E):=\{\gamma\in \cF(V,E)\colon \Supp(\gamma)\subseteq K\}$
is endowed with the topology induced by $\cF(V,E)$.
\item[Multiplication Axiom (MU):]
If $U\in\cU$
and $h\in C^\infty_c(U,\R)$,
then
$h\gamma\in\cF(U,E)$
for all
$\gamma\in \cF(U,E)$
and the map
\[
m_h^E\colon \cF(U,E)\to\cF(U,E),\quad
\gamma\mapsto h\gamma
\]
is continuous.
\end{description}
\end{defn}
\begin{rem}\label{only-scalar}
As the map in (\ref{an-iso})
is an isomorphism of
topological vector spaces,
we see that Axioms (PB), (GL), and (MU)
hold in general whenever they hold
for $E:=\R$.
Likewise,
Axiom (PF) holds in general
whenever it holds for $F:=\R$.
\end{rem}
\begin{rem}\label{auto-cont}
Concerning Axiom (MU), observe that if $\cF(U,E)$ is a Fr\'{e}chet space
and $h\gamma\in\cF(U,E)$
for each $\gamma\in\cF(U,E)$,
then $m_h^E$ is continuous.\\[2mm]
[The multiplication operator $M_h\colon BC(U,E)\to BC(U,E)$,
$\gamma\mapsto h\gamma$ being continuous,
its graph
$\graph(M_h)$
is closed in $BC(U,E)\times BC(U,E)$.
As the inclusion map $\iota\colon\cF(U,E)\to BC(U,E)$
is continuous, we deduce that
$(\iota\times \iota)^{-1}(\graph(M_h))=
\graph(m_h^E)$ is closed in $\cF(U,E)\times\cF(U,E)$.
The continuity of~$m_h^E$
now follows from the Closed Graph Theorem.\,]\\[2.3mm]
Likewise, continuity of the linear map
$\cF(V,E)\to\cF(W,E)$ in (PB)
is automatic if $\cF(V,E)$
and $\cF(W,E)$ are Fr\'{e}chet
spaces,
using that the linear map $BC(V,E)\to BC(W,E)$,
$\gamma\mto\gamma\circ\Theta|_W$
is continuous with operator norm $\leq 1$
(if we endow $E$ with a norm defining its topology
and spaces of bounded continuous
functions to~$E$
with the supremum norm).\\[2.3mm]
Likewise, $e^E_{U,V,K}$ is continuous
in (GL) if $\cF_K(V,E)$
and $\cF(U,E)$ are Fr\'{e}chet spaces.
In fact, endowing $BC_K(V,E):=\{\gamma\in BC(V,E)\colon \Supp(\gamma)\sub K\}$
with the supremum norm,
the map $BC_K(V,E)\to BC(U,E)$, $\gamma\mto \wt{\gamma}$
which extends functions by~$0$
is a linear isometry.
\end{rem}
\section{Basic consequences of the axioms}\label{cons-ax}
Let $m\in\N$,
$\cU$ be a good collection of subsets
of~$\R^m$
and $(\cF(U,\R))_{U\in\cU}$
be a family of locally convex spaces which is suitable
for Lie theory.
We record consequences of the four axioms.
\begin{la}\label{RES}
Let~$E$ be a finite-dimensional real vector space
and $U,W\in\cU$
such that $W$ is relatively compact in~$U$. 
Then $\gamma|_W\in\cF(W,E)$ holds
for each $\gamma\in\cF(U,E)$
and the restriction map
\[
r^E_{W,U}\colon \cF(U,E)\to \cF(W,E), \quad
\gamma\mapsto \gamma|_W
\]
is continuous.
\end{la}
\begin{proof}
We can take $V:=U$ and $\Theta:=\id_U$
in Axiom (PB).
\end{proof}
Lemma~\ref{RES}
and Axiom~(GL) imply:
\begin{la}\label{RES-K}
Let~$E$ be a finite-dimensional real vector space
and $U,W\in\cU$
such that $W$ is relatively compact in~$U$. 
Let $K\sub W$ be compact.
Then
\[
r^E_{W,U,K}\colon \cF_K(U,E)\to \cF_K(W,E), \quad
\gamma\mapsto \gamma|_W
\]
is an isomorphism of topological vector spaces. $\square$
\end{la}
Also the following maps are useful.
\begin{la}\label{Consequences_F}
Let $E$ and $F$ be finite-dimensional vector spaces.
Let $U,V\in\cU$ such that $V$ is relatively compact in~$U$,
and $\Phi\colon E\to F$ be a smooth map.
Then $\Phi\circ\gamma|_V\in\cF(V,F)$ holds
for each $\gamma\in\cF(U,E)$
and the mapping $\cF(U,E)\to \cF(V,F)$,
$\gamma\mapsto\Phi\circ\gamma|_V$ is continuous. 
\end{la}
\begin{proof}
Axiom (PF) applies to $f\colon U\times E\rightarrow F$,
$(x,y)\mapsto \Phi(y)$.
\end{proof}
\begin{defn}\label{def_F_loc}
Let $E$ be a finite-dimensional real vector space
and $U$ be an open subset of $\R^m$.
We let $\cF_\loc(U,E)$ be the set of
all functions $\gamma\colon U\to E$
with the following property:
For each $V\in\cU$
which is a relatively compact subset of~$U$, the restriction $\gamma|_V$ is in
$\cF(V,E)$. 
\end{defn}
Note that each $\gamma\in \cF_{\loc}(U,E)$
is continuous, and that $\cF_{\loc}(U,E)$
is a vector subspace of $E^U$.
If $U\in\cU$, then $\cF(U,E)\sub \cF_\loc(U,E)$,
by Lemma~\ref{RES}.
We give $\cF_\loc(U,E)$ the initial topology
with respect to the restrictions maps
\[
\rho^E_{V,U}\colon
\cF_\loc(U,E)\to \cF(V,E),\quad \gamma\mto\gamma|_V
\]
for all $V\in \cU$ which are relatively compact in~$U$.
As the restriction maps are linear
and separate points,
$\cF_\loc(U,E)$
is a Hausdorff locally convex space.
\begin{la}\label{loc-restr}
Let $E$ be a finite-dimensional vector space.
If $U$ and $V$ are open subsets of $\R^m$
such that $V\sub U$,
then $\gamma|_V\in\cF_{\loc}(V,E)$
for each $\gamma\in\cF_{\loc}(U,E)$
and the restriction map
\[
\delta^E_{V,U}\colon\cF_{\loc}(U,E)\to\cF_{\loc}(V,E),\quad\gamma\mto\gamma|_V
\]
is continuous and linear.
\end{la}
\begin{proof}
Each $W\in\cU$ which is relatively compact
in~$V$ is also relatively compact in~$U$,
whence $(\gamma|_V)|_W=\gamma|_W\in\cF(W,E)$
by definition of $\cF_{\loc}(U,E)$.
Hence $\gamma|_V\in \cF_{\loc}(V,E)$.
Since $\rho^E_{W,V}\circ \delta^E_{V,U}=r^E_{W,V}\circ \rho^E_{V,U}$
is continuous for each~$W$,
the map $\delta^E_{V,U}$ is continuous.
\end{proof}
\begin{la}\label{PB_loc}
Let $E$ and $F$ be finite-dimensional vector spaces and $U$ be an open subset of $\R^m$.
For each $\Phi\in C^\infty(E,F)$, the
assignment $\gamma\mapsto\Phi\circ\gamma$
defines a continuous map
$\cF_\loc(U,\Phi)\colon
\cF_\loc(U,E)\to\cF_\loc(U,F)$.
\end{la}
\begin{proof}
Let $\gamma\in\cF_\loc(U,E)$.
If $V\in \cU$ is relatively compact in~$U$,
we have $\wb{V}\sub W$ for some $W\in\cU$
which is relatively compact in~$U$,
by Definition~\ref{good-coll}\,(c).
Since $\gamma|_W=\rho^E_{W,U}(\gamma)\in\cF(W,E)$,
we have $\Phi\circ\gamma|_V=\Phi\circ(\gamma|_W)|_V\in\cF(V,F)$,
by Lemma~\ref{Consequences_F}.
Thus $\Phi\circ\gamma\in \cF_{\loc}(U,F)$.
For $V$ and $W$ as before,
the map
\[
h_V\colon \cF(W,E)\to \cF(V,F),\quad \eta\mto \Phi\circ \eta|_V
\]
is continuous, by Lemma~\ref{Consequences_F}.
Thus $\rho^F_{V,U}\circ \cF_{\loc}(U,\Phi)=h_V\circ \rho^E_{W,U}$
is continuous.
%
% Begruendung streichen?:
%
The topology on $\cF_{\loc}(U,F)$ being initial
with respect to the maps $\rho^F_{V,U}$,
we deduce that $\cF_{\loc}(U,\Phi)$
is continuous.
\end{proof}
We also need the following variant.
\begin{la}\label{pf-locloc}
Let $E$ and $F$ be finite-dimensional
vector spaces and $U$ be an open subset of $\R^m$.
Let $\Psi\colon Q\to F$ be a smooth function
on an open subset $Q\sub E$.
Then $\Psi\circ \gamma\in\cF_{\loc}(U,F)$
for each $\gamma\in\cF_{\loc}(U,E)$
such that $\gamma(U)\sub Q$.
\end{la}
\begin{proof}
Let $\gamma\in\cF_{\loc}(U,E)$ with $\gamma(U)\sub Q$.
For each $W\in\cU$ which is relatively compact in~$U$,
there exists $V\in\cU$ such that $\wb{W}\sub V$
and $V$ is relatively
compact in~$U$, by Definition~\ref{good-coll}\,(c).
Then $\gamma|_V\in\cF(V,E)$.
The image $K:=\gamma(\wb{V})$
is a compact subset of~$Q$.
There exists $\xi\in C^\infty_c(Q,\R)$
such that $\xi|_K=1$.
We define $\Phi(y):=\xi(y)\Psi(y)$
for $y\in Q$ and $\Phi(y):=0$
for $y\in E\setminus K$.
Then $\Phi\in C^\infty(E,F)$
and $\Phi(y)=\Psi(y)$ for each $y\in \gamma(V)$,
whence $(\Psi\circ\gamma)|_W=\Psi\circ(\gamma|_W)
=\Phi\circ (\gamma|_W)=(\Phi\circ (\gamma|_V))|_W
\in\cF(W,F)$ by Lemma~\ref{Consequences_F}.
Thus $\Psi\circ\gamma\in \cF(U,F)$.
\end{proof}
\begin{la}\label{right_loc}
Let $E$ be a finite-dimensional vector space, $U$ and~$V$ be
%bounded
open subsets of $\R^m$ and $\Theta\colon U\rightarrow V$ be a $C^\infty$-diffeomorphism.
Then $\gamma\circ \Theta$\linebreak
$\in \cF_\loc(U,E)$ holds
for all $\gamma\in \cF_\loc(V,E)$.
Moreover, the linear mapping\linebreak
$\cF_{\loc}(\Theta,E)\colon \cF_{\loc}(V,E)\to \cF_{\loc}(U,E)$,
$\gamma\mto\gamma\circ \Theta$ is continuous.
\end{la}
\begin{proof}
Let $W\in\cU$ be relatively compact in~$U$.
By Definition~\ref{good-coll}\,(c),
there exists $P\in\cU$
such that $\wb{P}$ is compact
and $\wb{W}\sub P\sub\wb{P}\sub U$.
Then $Q:=\Theta(P)$ is a relatively compact
subset of~$V$ and $Q\in\cU$ by Definition~\ref{good-coll}\,(d).
Hence $\gamma|_Q\in\cF(Q,E)$.
By Axiom~(PB), we have
\[
(\gamma\circ \Theta)|_W=(\gamma|_Q\circ \Theta|_P^Q)|_W\in \cF(W,E).
\]
Hence $\gamma\circ\Theta\in \cF_\loc(U,E)$.
By Axiom~(PB),
the map $\rho^E_{W,U}\circ \cF_{\loc}(\Theta,E)=
\cF((\Theta|_P^Q)|_W,E)\circ \rho^E_{Q,V}$
is continuous
for each $W$ (choosing $P$ and $Q$ for~$W$ as before).
Hence $\cF_{\loc}(\Theta,E)$
is continuous.
\end{proof}
\begin{la}\label{CS_loc}
Let $E$ be a finite-dimensional vector space, $U\in\cU$
and $h\in C^{\infty}_c(U,\R)$.
Then $h\gamma\in \mathcal{F}(U,E)$ for all $\gamma\in \mathcal{F}_{\loc}(U,E)$ and the
mapping
$\mu_h^E\colon \mathcal{F}_{\loc}(U,E)\rightarrow \mathcal{F}(U,E)$,
$\gamma\mapsto h\gamma $ is continuous.
\end{la}
\begin{proof}
By Definition~\ref{good-coll}\,(b),
there exists $V\in\cU$ with
$K:=\Supp(h)\sub V$ such that $V$ is relatively compact in~$U$.
Let $\gamma\in \mathcal{F}_{\loc}(U,E)$.
Then $\gamma|_V\in \mathcal{F}(V,E)$. By Axiom~(MU),
$(h\gamma)|_V=h|_V\gamma|_V
\in\mathcal{F}(V,E)$,
entailing that $(h\gamma)|_V\in\cF_K(V,E)$.
Consequently, $h\gamma=(h|_V\gamma|_V)\hspace*{.3mm}\wt{\;}
\in \mathcal{F}(U,E)$
and $h\gamma=(e^E_{U,V,K}\circ m_{h|_V}^E\circ \rho^E_{V,U})(\gamma)\in\cF(U,E)$
depends continuously on $\gamma\in \cF_{\loc}(U,E)$
by continuity of $\rho^E_{V,U}$,
Axiom~(MU), and Axiom~(GL).
\end{proof}
\begin{la}\label{PartUnit_loc}
Let $E$ be a finite-dimensional vector space, $U_1,\ldots,U_n$
be open subsets of~$\R^m$
and $\gamma_j\in \mathcal{F}_{\loc}(U_j,E)$
for $j\in\{1,\ldots, n\}$
such that
\[
\gamma_j\vert_{U_j\cap U_k}=\gamma_k\vert_{U_j\cap U_k}
\quad
\mbox{for all $j,k\in\{1,\ldots, n\}$.}
\]
If
$V\in\cU$ is relatively compact in
$U_1\cup \ldots\cup U_n$, then $\gamma\in \cF(V,E)$
holds
for the map $\gamma\colon V\rightarrow E$ defined piecewise
via $\gamma(x) := \gamma_j(x)$ for $x\in V\cap U_j$.
\end{la}
\begin{proof}
By Definition~\ref{good-coll}\,(c),
we find $W\in\cU$ which is relatively compact
in $U_1\cup\cdots\cup U_n$
and contains $\wb{V}$.
Since $\cU$ is a basis
for the topology of~$\R^m$,
using the compactness of~$\wb{V}$
we find $W_1,\ldots, W_\ell\in\cU$
with $\wb{V}\sub \bigcup_{i=1}^\ell W_i$
such that, for each $i\in\{1,\ldots,\ell\}$,
the set $W_i$ is relatively compact in $U_{j(i)}\cap W$
for some $j(i)\in\{1,\ldots, n\}$.
Let $h_1,\ldots,h_\ell,h_0$ be a $C^\infty$-partition of unity on
$\R^m$ subordinate to $W_1,\ldots,W_\ell$, $\R^m\setminus\overline{V}$.
For each $i\in\{1,\ldots, \ell\}$,
the support $L_i:=\Supp(h_i)\sub W_i$ of~$h_i$ in~$\R^m$
is compact, as
$W_i$ is relatively compact in~$\R^m$.
Now~$\gamma$ can be written in the form
\[
\gamma=\sum_{i=1}^\ell ((h_i|_{W_i}\gamma_{j(i)}|_{W_i})\hspace*{.3mm}\wt{\;}\hspace*{.5mm})|_V,
\]
where the tilde indicates the extension by $0$
to an element of $\cF(W,E)$.
Thus $\gamma=\sum_{i=1}^\ell (r^E_{V,W}\circ
e^E_{W,W_i,L_i}\circ m^E_{h_i|_{W_i}}\circ \rho^E_{W_i,U_{j(i)}})(\gamma_{j(i)})
\in \cF(V,E)$.
\end{proof}
\begin{rem}\label{glueing-cts}
Let $\cE$ be the vector subspace of $\prod_{j=1}^n\cF_{\loc}(U_j,E)$
given by the $n$-tuples $(\gamma_1,\ldots,\gamma_n)$
such that $\gamma_j|_{U_j\cap U_k}=\gamma_k|_{U_j\cap U_k}$
for all $j,k\in\{1,\ldots, n\}$.
Endow $\cE$ with the topology induced by $\prod_{k=1}^n\cF_{\loc}(U_k,E)$.
The final formula of the preceding proof
shows that the linear map
\[
\glue \colon \cE\to \cF(V,E),\quad (\gamma_1,\ldots,\gamma_n)\mto \gamma
\]
(with $\gamma$ as in Lemma~\ref{PartUnit_loc})
is continuous.
\end{rem}
\section{Associated function spaces on manifolds}
Let $m$, $\cU$ and $(\cF(U,\R))_{U\in\cU}$
be as in the preceding section.
\begin{defn}\label{f-to-N}
Let $M$ be a compact smooth manifold
of dimension~$m$ and~$N$ a smooth manifold of dimension~$n$,
both without boundary.
We let $\mathcal{F}(M,N)$
be the set of all functions
$\gamma\colon M\to N$ with the following property:
For each $x\in M$, there exists a chart
$\phi\colon M\supseteq U_\phi
\rightarrow V_\phi\subseteq \R^m$ of~$M$ with $V_\phi\in\cU$
and a chart $\psi\colon
N\supseteq U_{\psi} \rightarrow V_{\psi}\subseteq \R^n$ of~$N$ such that $x\in U_{\phi}$,
$\gamma(U_{\phi})\subseteq U_{\psi}$, and $
\psi\circ \gamma\circ\phi^{-1} \in \mathcal{F}(V_{\phi},\R^n)$.
\end{defn}
Since $\psi\circ \gamma\circ\phi^{-1}$
is continuous, $\gamma|_{U_\phi}$
is continuous in the preceding situation.
Hence each $\gamma\in \cF(M,N)$ is continuous.
\begin{la}\label{charact_F(M,E)}
Let $M$ be an $m$-dimensional compact smooth
manifold, $N$ be a smooth manifold of dimension~$n$ and $\gamma\colon M\rightarrow N$ be a continuous map.
Then $\gamma\in \mathcal{F}(M,N)$ if and only if $\psi\circ\gamma\circ\phi^{-1}
\in \mathcal{F}_{\loc}(V_{\phi},\R^n)$ for each chart $\phi\colon
M\supseteq U_{\phi} \rightarrow V_{\phi}
\subseteq \R^m$ of $M$ and each chart $\psi\colon N\supseteq U_{\psi}
\rightarrow V_{\psi}\subseteq \R^n$ of~$N$ such that $\gamma(U_{\phi})\subseteq U_{\psi}$.
\end{la}
\begin{proof}
If $\gamma\in\cF(M,N)$,
let $\phi\colon U_{\phi}
\rightarrow V_{\phi}\subseteq \R^m$ be a chart of~$M$
and $\psi\colon U_{\psi}\rightarrow V_{\psi}\subseteq \R^n$
be a chart of~$N$ such that $\gamma(U_{\phi})\subseteq U_{\psi}$.
Let $W\in\cU$ be relatively compact
in~$V_{\phi}$. By definition, for each point $x\in M$
there exist a chart $\phi_x\colon
U_x \rightarrow V_x\subseteq \R^m$ of $M$ with $V_x\in\cU$
and a chart
$\psi_x\colon A_x \rightarrow B_x\subseteq \R^n$ of~$N$
such that $x\in U_x$,
$\gamma(U_x)\subseteq A_x$ and $\psi_x\circ\gamma\circ\phi_x^{-1}
\in \mathcal{F}(V_x,\R^n)$.
Since $M =\bigcup_{x\in M}U_x$, there exists a finite subcover $U_{x_1},\ldots,
U_{x_r}$ of~$M$.
We have
\[
g_i\defi \psi_{x_i}\circ\gamma\circ\phi_{x_i}^{-1}
\in \mathcal{F}(V_{x_i},\R^n)\subseteq \mathcal{F}_{\loc}(V_{x_i},\R^n)
\]
for all $i\in \{1,\ldots,r\}$.
Then $g_i|_{\phi_{x_i}(U_\phi\cap U_{x_i})}
\in \mathcal{F}_{\loc}(\phi_{x_i}(U_{\phi}\cap U_{x_i}),\R^n)$,
since $\phi_{x_i}(U_{\phi}\cap U_{x_i})$ is an open subset of~$V_{x_i}$
(see Lemma~\ref{loc-restr}).
For each $i\in\{1,\ldots,r\}$,
consider the $C^\infty$-diffeomorphisms
\[
\Theta_i:= \phi_{x_i}\circ \phi^{-1}
\colon \phi (U_{\phi}\cap U_{x_i})\rightarrow \phi_{x_i}(U_{\phi}\cap U_{x_i})
\]
and $\Phi_i:= \psi\circ \psi^{-1}_{x_i}\colon \psi_{x_i} (U_{\psi}\cap A_{x_i})\rightarrow \psi(U_{\psi}\cap A_{x_i})$.
By Lemma~\ref{right_loc}, we have
\[
g_i\circ \Theta_i \in \mathcal{F}_{\loc}(\phi (U_{\phi}\cap U_{x_i}),\R^n),
\]
whence $\Phi_i\circ g_i\circ \Theta_i
\in \mathcal{F}_{\loc}(\phi (U_{\phi}\cap U_{x_i}),\R^n)$,
by Lemma~\ref{pf-locloc}.
Note that
\[
(\psi\circ f\circ
\phi^{-1})|_{\phi(U_\phi\cap U_{x_i})}
=\Phi_i\circ g_i\circ\Theta_i;
\]
% Argument streichen?:
in fact, for all $y\in \phi(U_\phi\cap U_{x_i})$, we have
\[
(\Phi_i\circ g_i\circ \Theta_i)(y)=((\psi\circ \psi^{-1}_{x_i})\circ(\psi_{x_i}\circ\gamma\circ \phi_{x_i}^{-1})
\circ (\phi_{x_i}\circ\phi^{-1}))(y)=(\psi\circ\gamma\circ \phi^{-1})(y).
\]
Since $\overline{W}\subseteq V_{\phi} = \bigcup\limits^r_{i=1}\phi (U_{\phi}\cap U_{x_i})$,
we are in
the situation of Lemma~\ref{PartUnit_loc}
and deduce that
$\psi\circ \gamma\circ \phi^{-1}|_W
\in \mathcal{F}(W,\R^n)$. Thus
$\psi\circ\gamma\circ\phi^{-1}\in \cF_\loc(V_\phi,\R^n)$.
\\[2.3mm]
Conversely, let $\gamma\colon M\to N$
be a function such that, for each $x\in M$, there are a chart
$\phi\colon U_{\phi}\rightarrow V_{\phi}\subseteq \R^m$
of~$M$ and a chart
$\psi\colon U_{\psi} \rightarrow V_{\psi}\subseteq \R^n$ of~$N$
such that $x\in U_{\phi}$ holds, $\gamma(U_{\phi})\subseteq U_{\psi}$ and $\psi\circ\gamma\circ\phi^{-1}\in \mathcal{F}_{\loc}(V_{\phi},\R^n)$. 
Since $\cU$ is a basis for the topology on~$\R^m$,
there exists $W\in\cU$ such that $\phi(x)\in W$
and $W$ is relatively compact
in~$V_{\phi}$.
Then $(\psi\circ\gamma\circ \phi^{-1}) |_W\in \mathcal{F}(W,\R^n)$.
Thus $\eta :=\phi |_W\colon \phi^{-1}(W)
\rightarrow W$ is a chart of~$M$
such that $x\in \phi^{-1}(W)$, $\gamma (\phi^{-1}(W))\subseteq U_{\psi}$
and
$\psi\circ\gamma\circ\eta^{-1}=(\psi\circ\gamma\circ\phi^{-1})|_W\in \mathcal{F}(W,\R^n)$.
Hence $\gamma\in\cF(M,N)$.
\end{proof}
\begin{la}\label{PF-between}
Let $\Phi\colon N_1\to N_2$
be a smooth map between finite-dimensional
smooth manifolds,
and $M$ be a compact $m$-dimensional
smooth manifold.
Then $\Phi\circ\gamma\in \cF(M,N_2)$
for each $\gamma\in\cF(M,N_1)$.
\end{la}
\begin{proof}
For each $x\in M$,
there exists a chart $\psi_2\colon U_2\to V_2\sub\R^{n_2}$
of $N_2$ with $\Phi(\gamma(x))\in U_2$
and a chart $\psi_1\colon U_1\to V_1\sub\R^{n_1}$
on~$N_1$ such that $\gamma(x)\in U_1$
and $\Phi(U_1)\sub U_2$.
Moreover, there exists a chart
$\phi\colon U\to V\sub\R^m$
for $M$ with $x\in U$ such that $\gamma(U)\sub U_1$.
By Definition~\ref{good-coll}\,(a),
there exists $W\in\cU$
such that $\phi(x)\in W$ and $W$
is relatively compact in~$V$.
Now $\psi_1\circ\gamma\circ\phi^{-1}\in\cF_{\loc}(V,\R^{n_1})$,
by Lemma~\ref{charact_F(M,E)}.
Since $\psi_2\circ \Phi\circ\psi_1^{-1}$
is smooth, Lemma~\ref{pf-locloc} shows that
$\psi_2\circ \Phi\circ \gamma\circ\phi^{-1}=
(\psi_2\circ\Phi\circ \psi_1^{-1})\circ (\psi_1\circ\gamma\circ\phi^{-1})\in\cF_{\loc}(V,\R^{n_2})$.
Hence $\psi_2\circ (\Phi\circ\gamma)\circ\phi^{-1}|_W\in\cF(W,\R^{n_2})$.
Thus $\Phi\circ\gamma\in\cF(M,N_2)$.
\end{proof}
As a special case of Definition~\ref{f-to-N}, taking $N:=E$
we defined $\mathcal{F}(M,E)$
whenever $M$ is an $m$-dimensional
compact smooth manifold and $E$ a
finite-dimensional real vector space.
\begin{numba}\label{the-ini-top}
We give $\mathcal{F}(M,E)$
the initial topology~$\cO$ with respect to the mappings
\[
\cF(M,E)\to \cF_{\loc}(V_\phi,E),\quad\gamma\mto \gamma\circ\phi^{-1},
\]
for $\phi\colon U_\phi\to V_\phi\sub\R^m$
in the maximal $C^\infty$-atlas $\cA$ of~$M$.
As the latter maps
are linear and separate points on $\cF(M,E)$,
the topology $\cO$ makes
$\cF(M,E)$ a Hausdorff locally convex space.
By transitivity of initial topologies
(see
%, e.g.,
\cite[Lemma~A.2.7]{GaN}), the topology $\cO$ is also
initial with respect to the maps
\[
\cF(\phi^{-1}|_W,E)\colon \cF(M,E)\to \cF(W,E),\quad \gamma\mto\gamma\circ\phi^{-1}|_W,
\]
for $\phi\colon U_\phi\to V_\phi\sub\R^m$
in $\cA$ and $W\in\cU$ such that $W$
is a relatively compact subset of~$V_\phi$.
\end{numba}
Finitely many pairs $(\phi,W)$
suffice to define the topology~$\cO$.
\begin{prop}\label{fin-many}
Let $\phi_j\colon U_j\to V_j\sub\R^m$
be charts for $M$
for $j\in \{1,\ldots,k\}$
and $W_j\in \cU$ such that
$W_j$ is relatively compact in~$V_j$
and $M=\bigcup_{j=1}^k\phi_j^{-1}(W_j)$.
Then the map
\[
\Theta\colon (\cF(M,E),\cO)\to \prod_{j=1}^k
\cF(W_j,E),\quad \gamma\mto (\gamma\circ\phi_j^{-1}|_{W_j})_{j=1}^k
\]
is linear and a topological embedding with closed
image. The image $\im(\Theta)$
is the set~$S$ of all $(\gamma_j)_{j=1}^k\in\prod_{j=1}^k\cF(W_j,E)$
such that
$\gamma_i(\phi_i(x))=\gamma_j(\phi_j(x))$
for all $i,j\in\{1,\ldots, k\}$
and $x\in \phi_i^{-1}(W_i)\cap\phi_j^{-1}(W_j)$.
\end{prop}
\begin{proof}
Let $\cT$ be the initial topology on $\cF(M,E)$
with respect to the maps $\cF(\phi_j^{-1}|_{W_j},E)
\colon \cF(M,E)\to
\cF(W_j,E)$, $\gamma\mto \gamma\circ\phi_j^{-1}|_{W_j}$
for $j\in \{1,\ldots, k\}$.
Then $\cT\sub\cO$.
To see that $\cO\sub \cT$
holds, we have to show that
$\cT$ makes $\cF(\phi^{-1}|_W,E)$
continuous for each chart $\phi\colon U\to V\sub\R^m$
of~$M$
and each $W\in \cU$
which is relatively compact in~$V$.
Abbreviate $Q_j:=\phi_j^{-1}(W_j)$ for $j\in\{1,\ldots, k\}$.
Let $\cE\sub \prod_{j=1}^k \cF_{\loc}(\phi(U\cap Q_j),E)$
be the set of those $(\gamma_j)_{j=1}^k$
with $\gamma_i(x)=\gamma_j(x)$
for all $x\in \phi(U\cap Q_i\cap Q_j)$.
Then $W\sub \bigcup_{j=1}^n \phi(U\cap Q_j)$.~Let
\[
\glue\colon \cE\to\cF(W,E)
\]
be the continuous linear glueing map,
as in Remark~\ref{glueing-cts}.
For $j\in\{1,\ldots,k\}$, consider the $C^\infty$-diffeomorphism
$\Theta_j:=\phi_j\circ \phi^{-1}\colon \phi(U\cap Q_j)\to
\phi_j(U\cap Q_j)$.
The inclusion map $\lambda_j\colon \cF(W_j,E)\to\cF_{\loc}(W_j,E)$
is continuous linear
and so is the map $\cF_\loc(\Theta_j,E)\colon \cF_{\loc}(\phi_j(U\cap Q_j),E)\to
\cF_{\loc}(\phi(U\cap Q_j),E)$.
Abbreviate $h_j:=\cF_{\loc}(\Theta_j,E)\circ \delta^E_{\phi_j(U\cap Q_j),W_j}\circ
\lambda_j\circ\cF(\phi_j^{-1}|_{W_j},E)$.
Then $h_j(\gamma)=(\gamma\circ\phi^{-1})|_{\phi(U\cap Q_j)}$
for all $\gamma\in \cF(M,E)$ and $j\in\{1,\ldots,k\}$,
entailing that $(h_j(\gamma))_{j=1}^k\in\cE$
and $\gamma\circ \phi^{-1}|_W=\glue((h_j(\gamma))_{j=1}^k)$.
Thus $\cF(\phi^{-1}|_W,E)=\glue\circ (h_1,\ldots, h_k)$
is continuous on $(\cF(M,E),\cT)$.\\[2.3mm]
By the preceding, $\cO=\cT$,
whence the linear map $\Theta$ is a topological embedding.
Since $\cF(W_j,E)\sub BC(W_j,E)$
and the inclusion map is continuous,
the point evaluation $\cF(W_j,E)\to E$, $\gamma\mto\gamma(x)$
is continuous for each $x\in W_j$.
As it is also linear,
we see that~$S$ is a closed vector subspace
of $\prod_{j=1}^k\cF(W_j,E)$.
We easily see that $\im(\Theta)\sub S$.
If $(\gamma_j)_{j=1}^k \in S$,
then $\gamma\colon M\to E$, $\gamma(x):=\gamma_j(\phi_j(x))$
if $x\in\phi^{-1}_j(W_j)=:Q_j$ with $j\in\{1,\ldots, k\}$
is well defined and continuous, as $\gamma|_{Q_j}
=\gamma_j\circ \phi_j|_{Q_j}$ is continuous.
Since $\eta_j:=\phi_j|_{Q_j}\colon Q_j\to W_j$
is a chart of~$M$ with range in~$\cU$
such that $\gamma\circ\eta_j=\gamma_j\in \cF(W_j,E)$,
going back to Definition~\ref{f-to-N}
we see that $\gamma\in \cF(M,E)$.
By construction, $\Theta(\gamma)=(\gamma_j)_{j=1}^k$.
Thus $S\sub \im(\Theta)$ and thus $S=\im(\Theta)$.
\end{proof}
We record an immediate consequence:
\begin{cor}
The locally convex space
$\cF(M,E)$ is integral complete.
For $(\phi_1,W_1)$, $\ldots$, $(\phi_k,W_k)$
as in Proposition~{\rm\ref{fin-many}},
we have:
\begin{itemize}
\item[\rm(a)]
If $\cF(W_j,E)$
is a Banach space for all
$j\in\{1,\ldots, k\}$,
then $\cF(M,E)$
is a Banach space.
\item[\rm(b)]
If $\langle.,.\rangle_j$
is a scalar product
on $\cF(W_j,E)$
such that the associated norm
defines its topology
and makes it a Hilbert space,
then
\[
\cF(M,E)\times\cF(M,E)\to\R ,\;\;
(\gamma,\eta)\mto \langle \gamma,\eta\rangle:=\!\!
\sum_{j=1}^k\langle \gamma\circ\phi_j^{-1}|_{W_j},
\eta\circ\phi_j^{-1}|_{W_j}\rangle
\]
is a scalar product making $\cF(M,E)$
a Hilbert space whose associated
norm defines the given topology~$\cO$
on $\cF(M,E)$. $\,\square$
\end{itemize}
\end{cor}
\begin{la}
The inclusion map
$\lambda_M^E\colon \mathcal{F}(M,E)\to C(M,E)$, $\gamma\mto \gamma$
is continuous.
For each $x\in M$, the point evaluation
\[
\ev_x\colon \cF(M,E)\to E,\quad \gamma\mto \gamma(x)
\]
is continuous and linear.
Let
$\cF(M,U):=\{\gamma\in \cF(M,E)\colon \gamma(M)\sub U\}$
for an open subset $U\sub E$.
Then $\cF(M,U)$ is an open subset of $\cF(M,E)$.
\end{la}
\begin{proof}
For a norm $\|\cdot\|$ on~$E$,
the topology on $C(M,E)$
can be obtained by the corresponding supremum norm~$\|\cdot\|_\infty$.
Using $(\phi_1,W_1)$, $\ldots$, $(\phi_k,W_k)$
as in Proposition~\ref{fin-many}
and $Q_j:=\phi_j^{-1}(W_j)$,
we have
\[
\|\gamma\|_\infty=\max_{j=1,\ldots, k}
\|\gamma|_{Q_j}\|_\infty
=\max_{j=1,\ldots,k}\|\gamma\circ \phi_j^{-1}|_{W_j}\|_\infty
=\max_{j=1,\ldots,k}\|\cF(\phi_j^{-1}|_{W_j},E)(\gamma)\|_\infty.
\]
As $\cF(\phi_j^{-1}|_{W_j},E)\colon \cF(M,E)\to\cF(W_j,E)$
is continuous linear and the supremum norm
on $\cF(W_j,E)$ is continuous,
we see that $\|\lambda^E_M(\gamma)\|_\infty=\|\gamma\|_\infty$ is continuous
in $\gamma\in\cF(M,E)$,
entailing that the linear map $\lambda^E_M$
is continuous.
Since $C(M,U)$ is open in $C(M,E)$,
we deduce that $\cF(M,U)=(\lambda^E_M)^{-1}(C(M,U))$
is open in $\cF(M,E)$.
Finally, use that point evaluations
on $C(M,E)$ are continuous.
\end{proof}
\begin{la}\label{PF_cont_Manifold}
If $M$ is a compact smooth manifold of dimension~$m$, $E$ and~$F$ are finite-dimensional
vector spaces,
$U\sub E$ is open and $f\colon M\times U\rightarrow F$ a smooth mapping, then
$f_*(\gamma):=f\circ (\id_M,\gamma)\in\cF(M,F)$
for each $\gamma\in \cF(M,U)\sub \cF(M,E)$
and the map
$f_*\colon \mathcal{F}(M,U)\rightarrow \mathcal{F}(M,F)$
is continuous.
\end{la}
\begin{proof}
Let $\gamma\in \cF(M,U)$.
Then $\gamma(M)$ is a compact subset
of~$U$,
and we find an open, relatively compact subset
$V\sub U$ such that $\gamma(M)\sub V$.
There exists a smooth function $\xi\colon U\to\R$
with compact support $K:=\Supp(\xi)\sub U$
such that $\xi|_V=1$.
We define a function $g\colon M\times E\to F$ piecewise
via $g(x,y):=\xi(y)f(x,y)$ if $(x,y)\in M\times U$
and $g(x,y):=0$ if $(x,y)\in M\times (E\setminus K)$.
Then $g$ is smooth.
Let $k\in\N$, charts $\phi_j\colon U_j\to V_j\sub\R^m$
for~$M$ and sets $W_j\in\cU$
for $j\in\{1,\ldots, k\}$
be as in Proposition~\ref{fin-many}.
By Definition~\ref{good-coll}\,(c),
for each $j\in \{1,\ldots, k\}$
there exists $Y_j\in\cU$ such that
$Y_j$ is relatively compact in~$V_j$
and $\wb{W}_j\sub Y_j$.
The map
\[
g_j\colon Y_j\times E\to F,\quad (x,y)\mto g(\phi_j^{-1}(x),y)
\]
is smooth, for each $j\in\{1,\ldots, k\}$.
By Axiom (PF),
we have $(g_j)_*(\eta):=g_j\circ (\id_{W_j},\eta|_{W_j})\in\cF(W_j,F)$
for each $\eta\in \cF(Y_j,E)$ and the map
\[
(g_j)_*\colon \cF(Y_j,E)\to \cF(W_j,F)
\]
is continuous. For each $\eta\in \cF(M,E)$,
we have
\[
g_*(\eta)\circ\phi_j^{-1}|_{W_j}=g\circ (\id_M,\eta) \circ \phi_j^{-1}|_{W_j}
=(g_j)_*(\eta\circ \phi_j^{-1}|_{Y_j})\in \cF(W_j,F),
\]
whence $g_*(\eta)\in \cF(M,F)$.
Since $\cF(\phi_j^{-1}|_{W_j},F)\circ g_*
=(g_j)_*\circ \cF(\phi_j^{-1}|_{Y_j},E)$
is continuous for $j\in\{1,\ldots,k\}$,
Proposition~\ref{fin-many} shows that $g_*$
is continuous. For each $\eta$ in the open
neighbourhood $P:=\cF(M,V)$ of $\gamma$ in $\cF(M,U)$,
we have $f_*(\eta)=g_*(\eta)\in \cF(M,F)$.
Notably, $f_*|_P=g_*|_P$ is continuous,
%
% koennte weglassen:
%
whence $f_*$ is continuous at~$\gamma$.
\end{proof}
\begin{la}\label{super-mfd-cts}
If $E$ and $F$ are finite-dimensional vector spaces
and $\Phi\colon U\to F$
is a smooth function on an open subset $U\sub E$,
then $\Phi\circ\gamma\in \cF(M,F)$ for each $\gamma\in \cF(M,U)\sub \cF(M,E)$
and the map
$\cF(M,\Phi)\colon \cF(M,U)\to\cF(M,F)$, $\gamma\mto\Phi\circ\gamma$
is continuous.
\end{la}
\begin{proof}
Lemma~\ref{PF_cont_Manifold} applies as
$\Phi\circ \gamma=f_*(\gamma)$
for the $C^\infty$-map $f\colon M\times U\to F$, $(x,y)\mto\Phi(y)$.
\end{proof}
Mappings to products correspond to pairs
of mappings.
\begin{la}\label{in-product}
If $E_1$ and $E_2$ are finite-dimensional
vector spaces, we consider
the projections $\pr_j\colon E_1\times E_2\to E_j$, $(x_1,x_2)\mto x_j$
for $j\in \{1,2\}$.
Then
\[
\Xi:=(\cF(M,\pr_1),\cF(M,\pr_2))\colon \cF(M,E_1\times E_2)\to
\cF(M,E_1)\times \cF(M,E_2)
\]
is an isomorphism of topological vector spaces.
\end{la}
\begin{proof}
The mappings $\lambda_1\colon E_1\to E_1\times E_2$,
$x_1\mto (x_1,0)$ and $\lambda_2\colon E_2\to E_1\times E_2$, $x_2\mto (0,x_2)$
are continuous and linear, as well as $\pr_1$, $\pr_2$ and the projections
\[
\pi_j\colon \cF(M,E_1)\times \cF(M,E_2)\to\cF(M,E_j),\quad
(\gamma_1,\gamma_2)\mto \gamma_j
\]
for $j\in\{1,2\}$.
By Lemma~\ref{super-mfd-cts},
$\Xi$ and
the linear map
\[
\Theta:=\cF(M,\lambda_1)\circ\pi_1+\cF(M,\lambda_2)\circ\pi_2\colon
\cF(M,E_1)\times\cF(M,E_2)\to\cF(M,E_1\times E_2)
\]
are continuous. We readily check that $\Xi\circ \Theta$
and $\Theta\circ\Xi$ are the identity
maps, whence $\Xi$ is an isomorphism
of topological vector spaces with $\Xi^{-1}=\Theta$.
\end{proof}
\begin{rem}\label{loc-prod}
Likewise,
using Lemma~\ref{PB_loc} instead of Lemma~\ref{super-mfd-cts}
we see that
$\cF_{\loc}(U,E_1\times E_2)\cong
\cF_{\loc}(U,E_1)\times \cF_{\loc}(U,E_2)$
for all open subsets $U\sub\R^m$
and finite-dimensional vector spaces $E_1$
and~$E_2$.
\end{rem}
\begin{la}\label{prod-deco}
Let $N_1$ and $N_2$ be finite-dimensional
smooth manifolds and
$M$ be a compact $m$-dimensional
smooth manifold.
Then $\cF(M,N_1\times N_2)=\cF(M,N_1)\times \cF(M,N_2)$,
identifying functions to $N_1\times N_2$
with the pair of components.
\end{la}
\begin{proof}
For $j\in\{1,2\}$,
the map $\pi_j\colon N_1\times N_2\to N_j$, $(x_1,x_2)\mto x_j$
is smooth. Hence, if $\gamma\in \cF(M,N_1\times N_2)$,
then $\pi_j\circ \gamma\in \cF(M,N_j)$,
by Lemma~\ref{PF-between}.
Conversely, let $\gamma_j\in\cF(M,N_j)$
for $j\in\{1,2\}$. Then the map
$\gamma:=(\gamma_1,\gamma_2)\colon$
$M\to N_1\times N_2$ is
continuous.
For $x\in M$,
there are charts $\psi_j\colon U_j\to V_j\sub \R^{n_j}$
of $N_j$ for $j\in \{1,2\}$
such that $\gamma_j(x)\in U_j$.
There exists a chart $\phi\colon U\to V\sub\R^m$
of~$M$ such that $x\in U$ and $\gamma(U)\sub U_1\times U_2$.
There exists $W\in \cU$
such that $\phi(x)\in W$ and
$W$ is relatively compact in~$V$.
Then $h:=(\psi_1\times\psi_2)
\circ\gamma\circ \phi^{-1}\in\cF_{\loc}(V,\R^{n_1}\times\R^{n_2})$
(using Remark~\ref{loc-prod}),
as its components
$\psi_j\circ \gamma_j\circ\phi^{-1}$ are in $\cF_{\loc}(V,\R^{n_j})$,
by Lemma~\ref{charact_F(M,E)}.
Thus $h|_W\in \cF(W,\R^{n_1}\times\R^{n_2})$.
As a consequence, $\gamma\in \cF(M,N_1\times N_2)$.
\end{proof}
We need a more technical variant of the above maps.
\begin{la}\label{for-Ad}
Let $E$ and $F$ be finite-dimensional vector spaces,
$N$ be a finite-dimensional smooth manifold,
$M$ be a  compact $m$-dimensional
smooth manifold,
$g\colon N\times E\to F$ be a $C^\infty$-map
and $\gamma\in \cF(M,N)$.
Define $f\colon M\times E\to F$ via
$f(x,y):=g(\gamma(x),y)$.
Then $f_*(\eta):=f\circ (\id_M,\eta)=g\circ (\gamma,\eta)\in\cF(M,F)$
for each $\eta\in \cF(M,E)$ and the map
$f_*\colon \cF(M,E)\to\cF(M,F)$
is continuous.
\end{la}
\begin{proof}
Since~$M$ is compact,
we find $k\in\N$ and charts $\phi_j\colon U_j\to V_j\sub\R^m$
of $M$ for $j\in\{1,\ldots, k\}$,
and relatively compact subsets $W_j\sub V_j$
such that $W_j\in\cU$
and $M=\bigcup_{j=1}^kQ_j$ with $Q_j:=\phi_j^{-1}(W_j)$,
such that $\gamma(U_j)\sub P_j$
for some chart $\psi_j\colon A_j\to B_j\sub\R^n$
of~$N$
and a relatively compact, open subset $P_j\sub A_j$.
Let $\xi_j\colon B_j\to\R$ be a smooth map with compact support
$K_j:=\Supp(\xi_j)$ such that $\xi_j|_{\psi_j(P_j)}=1$.
Then $g_j(x,y):=\xi_x(x)f(\psi_j^{-1}(x),y)$
for $(x,y)\in B_j\times E$,
$g_j(x,y):=0$ for $(x,y)\in (\R^n\setminus K_j)\times E$
defines a smooth map $g_j\colon \R^n\times E\to F$.
By Lemma~\ref{PB_loc}, the map
\[
\cF_{\loc}(V_j,g_j)\colon \cF_{\loc}(V_j,\R^n\times E)\to
\cF_{\loc}(V_j,F),\quad \theta\mto g_j\circ\theta
\]
is continuous. Identifying $\cF_{\loc}(V_j,\R^n\times E)$
with $\cF_{\loc}(V_j,\R^n)\times\cF_{\loc}(V_j,E)$,
we deduce from Lemma~\ref{charact_F(M,E)} that
\begin{eqnarray*}
(g\circ (\gamma,\eta))\circ\phi_j^{-1}|_{W_j}
&=& (g_j\circ (\psi_j\circ\gamma\circ\phi_j^{-1},\eta\circ \phi_j^{-1}))|_{W_j}\\
&= & (\rho^F_{W_j,V_j}\circ \cF_{\loc}(V_j,g_j))(\psi_j\circ \gamma\circ\phi_j^{-1},
\eta\circ\phi_j^{-1})
\end{eqnarray*}
is in $\cF(W_j,F)$.
Hence $g\circ (\gamma,\eta)\in \cF(M,F)$.
Now $h_j\colon \cF(M,E)\to\cF_{\loc}(V_j,E)$,
$\eta\mto\eta\circ\phi_j^{-1}$
is continuous by definition
of the topology on $\cF(M,E)$.
Using Lemma~\ref{PB_loc},
we deduce that the map
$\cF(M,E)\to\cF(W_j,F)$,
\[
\eta\mto f_*(\eta)\circ\phi_j^{-1}|_{W_j}
=(\rho^F_{W_j,V_j}\circ \cF_{\loc}(V_j,g_j))(\psi_j\circ \gamma\circ\phi_j^{-1},
h_j(\eta))
\]
is continuous. Hence $f_*$ is continuous, by Proposition~\ref{fin-many}.
\end{proof}
\begin{prop}\label{PF_smooth_Manifold}
If $M$ is an $m$-dimensional compact smooth manifold,
$E$ and~$F$ are finite-dimensional vector spaces, $U$ is an open subset of $E$ and
$f\colon M\times U\rightarrow F$ is a $C^\infty$-map, then also the map $f_*\colon
\mathcal{F}(M,U)\rightarrow \mathcal{F}(M,F)$, $\gamma\mapsto f\circ(\id_M,\gamma)$
is smooth.
\end{prop}
\begin{proof}
By Lemma~\ref{PF_cont_Manifold}, the map~$f_*$ is continuous.
We show by induction on $k\in\N$ that $f_*$
is $C^k$ for all $E$, $F$, $U$ and $f$
as in the proposition.
To see that $f_*$ is $C^1$,
let $\gamma\in \mathcal{F}(M,U)$
and $\eta\in \mathcal{F}(M,E)$.
We claim that the directional derivative $d(f_*)(\gamma,\eta)$
exists
and
\[
d(f_*)(\gamma,\eta)(x)=d_2f(x,\gamma(x),\eta(x))
\]
holds for all $x\in M$;
here $d_2f\colon M\times U\times E\to F$
is the smooth mapping\linebreak
$(x,y,z)\mto d(f_x)(y,z)$
with $f_x:=f(x,\cdot)\colon U\to F$.
If this is true, then
\[
d(f_*)(\gamma,\eta)=(d_2f)_*(\gamma,\eta)
\]
if we identify $\cF(M,E)\times \cF(M,E)$
with
$\cF(M,E\times E)$ by means of the isomorphism
of topological vector spaces (and hence $C^\infty$-diffeomorphism)
described in Lemma~\ref{in-product}.
The map $(d_2f)_*\colon \cF(M,U\times E)\to\cF(M,F)$
being continuous, $f$ is~$C^1$.
If $k\geq 2$, then $(d_2f)_*$
is $C^{k-1}$ by
induction and thus~$f$
is~$C^k$.\\[2.3mm]
Proof of the claim.
As $\gamma(M)$ and $\eta(M)$ are compact in $U$ and $E$, respectively, there is $\ve>0$
with
$\gamma(M)+\,]{-\ve},\ve[\, \eta(M)\subseteq U$.
The map $(d_2f)_*$ being continuous, also
\[
g\colon [0,1]\times \,]{-\ve},\ve[\,\to \cF(M,F),\quad
(s,t)\mto d_2f\circ (\id_M,\gamma+st\eta,\eta)
\]
is continuous. Since $\cF(M,F)$ is integral complete,
the weak integral
\[
h(t):=\int_0^1 g(s,t)\,ds
\]
exists in $\cF(M,F)$ for all $t\in \;]{-\ve},\ve[$.
By continuity of parameter-dependent integrals
(see \cite[Lemma~1.1.11]{GaN}),
$h\colon ]{-\ve},\ve[\;\to\cF(M,F)$
is continuous.
For $t\in \;]{-\ve},\ve[\,\setminus\{0\}$,
consider the difference quotient
\[
\Delta(t):=\frac{1}{t}(f_*(\gamma+t\eta)-f_*(\gamma)).
\]
For $x\in M$, let $\ev_x\colon \cF(M,F)\to F$
be the continuous linear point evaluation
at~$x$. Since
weak integrals and continuous linear maps commute \cite[Exercise 1.1.3\,(a)]{GaN},
using the Mean value Theorem \cite[Proposition 1.2.6]{GaN}
we see that
\begin{eqnarray*}
\ev_x(\Delta(t)) &= &\Delta(t)(x)=
\frac{1}{t}\big(f(x,\gamma(x)+t\eta(x))-f(x,\gamma(x))\big)\\
&=& \int_0^1 d_2f(x,\gamma(x)+st\eta(x),\eta(x))\,ds\\
&=& \ev_x\left(\int_0^1 d_2f\circ (\id_M,\gamma+st\eta,\eta)\,ds\right)
=\ev_x(h(t)).
\end{eqnarray*}
As the point evaluations separate points, we deduce that
$\Delta(t)=h(t)$, which converges to
$h(0)=\int_0^1d_2f\circ (\id_M,\gamma,\eta)\, ds=d_2f\circ(\id_M,\gamma,\eta)$
as $t\to 0$.
\end{proof}
Setting $f(x,y):=\Phi(y)$, we deduce:
\begin{cor}\label{left_comp_smooth_Manifold}
If $M$ is an $m$-dimensional compact smooth
manifold, $E$ and~$F$ are finite-dimensional vector spaces,
$U$ is an open subset of~$E$ and $\Phi\colon U\to F$
is smooth, then also the map $\cF(M,\Phi)\colon
\mathcal{F}(M,U)\rightarrow \mathcal{F}(M,F)$, $\gamma\mapsto\Phi\circ\gamma $
is smooth. $\,\square$
\end{cor}
\begin{numba}
If $E$ is a finite-dimensional complex vector space with $\C$-basis
$b_1,\ldots, b_n$, then $b_1,\ldots,b_n, ib_1,\ldots, i b_n$
is an $\R$-basis for~$E$ and $E=F\oplus iF$
as a real vector space, using the real span $F$ of
$b_1,\ldots,b_n$.
For each $U\in\cU$, we then have
\[
\cF(U,E)=\cF(U,F)\oplus i\cF(U,F)
\]
as a real vector space and we easily
check that the operation
\[
\C\times \cF(U,E)\to\cF(U,E), \;\;
(t+is)(\gamma+i\eta):=
(t\gamma-s\eta)+i(s\gamma+t\eta)
\]
makes $\cF(U,E)$
a complex locally convex space.
As in the real case,
the complex topological vector space
structure is independent of the
basis.
\end{numba}
\begin{numba}
If $E$ is a finite-dimensional complex vector space,
then the
mapings\linebreak
$\cF(\phi^{-1}|_W,E)\to\cF(W,E)$
are complex linear in the situation of \ref{the-ini-top}.
Hence~\ref{the-ini-top}
provides a complex locally convex vector space
structure on $\cF(M,E)$.
\end{numba}
\begin{cor}\label{superpo-cx}
If $E$ and $F$ are $\K$-vector spaces
for $\K\in\{\R,\C\}$
in the situation of Corollary~{\rm\ref{left_comp_smooth_Manifold}}
and $\Phi$ is $\K$-analytic,
then also the mapping\linebreak
$\cF(M,\Phi)\colon \cF(M,U)\to\cF(M,F)$,
$\gamma\mto\Phi\circ \gamma$ is $\K$-analytic.
\end{cor}
\begin{proof}
If $\K=\C$,
define $f\colon M\times U\to F$ via $f(x,y):=\Phi(y)$.
We know that $\cF(M,\Phi)=f_*$
is smooth over~$\R$
with directional derivatives
$d(f_*)(\gamma,\eta)=(df_2)_*(\gamma,\eta)=(d\Phi)\circ (\gamma,\eta)$.
As the latter are complex linear in $\eta$
for fixed~$\gamma$,
the map $f_*$ is complex analytic by~\ref{real-to-c}.\\[2mm]
If $\K=\R$, then $\Phi$ has a $\C$-analytic
extension
$g\colon V\to F_\C$
for some open subset $V\sub E_\C$ with $U\sub V$.
Since $\cF(M,g)$ is a $\C$-analytic extension
for $\cF(M,\Phi)$ which is defined on an open subset
in $\cF(M,E_\C)=\cF(M,E)_\C$ and takes values in
$\cF(M,F_\C)=\cF(M,F)_\C$,
the map $\cF(M,\Phi)$ is $\R$-analytic.
\end{proof}
\section{The Lie groups {\boldmath$\mathcal{F}(M,G)$}}\label{sec-groups}
To prove Proposition~\ref{fi-prop},
let $m$, $\cU$, and $(\cF(U,\R))_{U\in\cU}$
be as in Section~\ref{cons-ax}.
Let $M$ be a compact $C^\infty$-manifold of dimension~$m$
and $G$ be a finite-dimensional Lie group
over $\K\in\{\R,\C\}$,
with Lie algebra $\cg$.
Let $\mu_G\colon G\times G\to G$
be the group multiplication
and $\eta_G\colon G\to G$, $g\mto g^{-1}$
be the inversion map.
Then $\cF(M,G)$
is a subgroup of the group $G^M$
of all mappings $M\to G$.
In fact,
\[
\gamma_1\gamma_2:=\mu_G\circ (\gamma_1,\gamma_2)=\cF(M,\mu_G)(\gamma_1,\gamma_2)\in\cF(M,G)
\]
for all $\gamma_1,\gamma_2\in \cF(M,G)$,
by Lemmas~\ref{PF-between} and \ref{prod-deco}.
Likewise, $\gamma^{-1}:=\eta_G\circ\gamma=\cF(M,\eta_G)(\gamma)\in\cF(M,G)$
for each $\gamma\in \cF(M,G)$, by Lemma~\ref{PF-between}.
We now give $\cF(M,G)$ a $\K$-analytic manifold
structure as described in Proposition~\ref{fi-prop}.\\[2.3mm]
There exists a balanced open $0$-neighbourhood
$Q\sub \cg$ such that $P:=\exp_G(Q)$ is open
in~$G$ and $\phi:=\exp_G|_Q^P\colon Q\to P$
is a $\K$-analytic diffeomorphism.
There exists a balanced open $0$-neighbourhood
$V\sub Q$ such that $U:=\exp_G(V)$
satisfies $UU\sub P$.
Since $V=-V$, we have $U=U^{-1}$.
Lemma~\ref{PF-between}
implies that $\phi\circ \gamma\in \cF(M,P)$
for each $\gamma\in\cF(M,Q)$
and that
\[
\Theta:=\cF(M,\phi)\colon \cF(M,Q)\to\cF(M,P),\quad \gamma\mto \phi\circ\gamma
\]
is a bijection (with inverse $\eta\mto \phi^{-1}\circ\eta$).
We give $\cF(M,P)$
the $\K$-analytic manifold structure
modelled on $\cF(M,\cg)$
which turns $\cF(M,\phi)$
into a $\K$-analytic diffeomorphism.
Then $\cF(M,U)$ is open in $\cF(M,P)$,
as $\cF(M,V)$ is open in $\cF(M,Q)$.
Since $f\colon U\times U\to P$, $(x,y)\mto xy^{-1}$
is $\K$-analytic, also
\[
g\colon V\times V\to Q,\quad (x,y)\mto\phi^{-1}(\phi(x)\phi(y)^{-1})
\]
is $\K$-analytic, whence
\[
\cF(M,g)\colon \cF(M,V\times V)=\cF(M,V)\times\cF(M,V)
\to \cF(M,Q)
\]
is $\K$-analytic, by Corollary~\ref{superpo-cx}.
As
\[
\cF(M,f)=\Theta\circ \cF(M,g)\circ (\Theta^{-1}\times\Theta^{-1})|_{\cF(M,U)
\times\cF(M,U)},
\]
also the map $\cF(M,f)$ is $\K$-analytic,
which takes $(\gamma,\eta)\in\cF(M,U)\times \cF(M,U)$
to $f\circ(\gamma,\eta)=\gamma\eta^{-1}\in\cF(M,P)$.
We now use that the adjoint action
$\Ad\colon G\times \cg\to \cg$, $(g,y)\mto\Ad_g(y)$
is smooth. Given $\gamma\in \cF(M,G)$,
consider the inner automorphism
$I_\gamma\colon \cF(M,G)\to\cF(M,G)$, $\eta\mto\gamma\eta\gamma^{-1}$.
We deduce with Lemma~\ref{for-Ad} that
$\Ad\circ(\gamma,\eta)\in \cF(M,\cg)$
for all $\eta\in\cF(M,\cg)$ and that the linear map
\[
\beta \colon \cF(M,\cg)\to\cF(M,\cg),\quad \eta\mto \Ad\circ \hspace*{.3mm}(\gamma,\eta)
\]
is continuous (and hence $\K$-analytic).
Thus
$W:=\beta^{-1}(\cF(M,V))\cap \cF(M,V)$ is an open $0$-neighbourhood
in $\cF(M,V)$ such that $\beta(W)\sub \cF(M,V)$.
% Moreover
Also,
$\Theta(W)$ is open in $\cF(M,P)$.
As $\gamma(x)\exp_G(\eta(x))\gamma(x)^{-1}=\exp_G(\Ad_{\gamma(x)}(\eta(x)))$
for all $\eta\in W$ and $x\in M$, we have
\[
I_\gamma\circ \Theta|_W=\Theta\circ\beta|_W,
\]
whence $I_\gamma(\Theta(W))\sub \cF(M,P)$
and $I_\gamma|_{\Theta(W)}\colon \Theta(W)\to\cF(M,P)$
is $\K$-analytic.
By the familiar local description of
Lie group structures,
there is a uniquely
determined $\K$-analytic manifold structure
on $\cF(M,G)$ which is modelled on $\cF(M,\cg)$,
turns $\cF(M,G)$ into a $\K$-analytic Lie group,
and such that $\cF(M,U)$ is open in
$\cF(M,G)$ and the latter
induces the given $\K$-analytic manifold
structure thereon
(see Proposition~18 in \cite[Chapter~III, \S1, no.\,9]{Bou},
whose hypothesis that the modelling space be Banach
is not needed in the proof).\\[2.3mm]
By construction, $\Phi:=\cF(M,\phi|_V)\colon \cF(M,V)\to\cF(M,U)$
is a $\K$-analytic diffeomorphism
onto an open identity neighbourhood in the Lie group $\cG:=\cF(M,G)$.
Identifying $T_0\cF(M,V)=\{0\}\times\cF(M,\cg)$
with $\cF(M,G)$ via $(0,v)\mto v$,
we obtain an isomorphism
\[
\alpha:=T_0\Phi\colon \cF(M,\cg)\to T_e\cG
\]
of topological vector spaces.
Let $[.,.]_\cg$ be the Lie bracket on~$\cg$
and $[.,.]$
be the Lie bracket on $\cF(M,\cg)$
making $\alpha$ an isomorphism
of Lie algebras to $L(\cG)$.
Then $[.,.]$ is the pointwise
Lie bracket, i.e.,
\[
[\gamma,\eta](x)=[\gamma(x),\eta(x)]_\cg\quad\mbox{for all
$\gamma,\eta\in\cF(M,\cg)$ and $x\in M$.}
\]
To see this, consider the point evaluations
$\ev_x\colon \cG\to G$, $\gamma\mto \gamma(x)$
and $\ve_x\colon \cF(M,\cg)\to\cg$, $\gamma\mto\gamma(x)$
at $x\in M$. Since $\ev_x\circ \,\Phi=\exp_G\circ \,\ve_x$,
the homomorphism $\ve_x$ is $\K$-analytic on some identity neighbourhood
and hence $\K$-analytic.
We also deduce that
$T_e(\ev_x)\circ T_0\Phi=T_0\exp_G\circ \, T_0\ve_x$
and thus
\begin{equation}\label{alph-inv}
L(\ev_x)\circ\alpha=\ve_x
\end{equation}
with $L(\ev_x):=T_e(\ev_x)$.
As a consequence, $[\gamma,\eta](x)$ equals $\ve_x([\gamma,\eta])=L(\ev_x)([\alpha(\gamma),
\alpha(\eta)])=[L(\ev_x)(\alpha(\gamma)),L(\ev_x)(\alpha(\eta))]_\cg=[\gamma(x),\eta(x)]_\cg$
for all $\gamma,\eta\in\cF(M,\cg)$.
Note that $(L(\ev_x)(\alpha(\gamma)))_{x\in M}=\gamma$ by~(\ref{alph-inv}),
whence the inverse map
$\alpha^{-1}\colon L(\cG)\to \cF(M,\cg)$
is given by $\alpha^{-1}(v)=(L(\ev_x)(v))_{x\in M}$.
We claim:
%that
\[
\cF(M,\exp_G)\circ\alpha^{-1}\colon L(\cG)\to G,\quad v\mto \exp_G\circ \, \alpha^{-1}(v)
\]
\emph{is the exponential function $\exp_\cG$ of~$\cG$.} Since
$\cF(M,\exp_G)$ is a local $\K$-analytic diffeomorphism
at~$0$ (as it coincides with $\Phi$
on some $0$-neighbourhood),
then also $\exp_\cG$ will be a local $\K$-analytic
diffeomorphism at~$0$,
and thus~$\cG$ is a BCH-Lie group.
To prove the claim
%
% Potential fuer kuerzen:
%
and complete the proof of Proposition~\ref{fi-prop},
let $v\in L(\cG)$
and abbreviate $\gamma:=\alpha^{-1}(v)$.
Then
\[
c\colon \R\to \cG,\quad t\mto \exp_G\circ \hspace*{.3mm}(t\gamma)
\]
is a homomorphism of groups and smooth as
it coincides with the smooth map $t\mto\Phi(t\gamma)$
on some $0$-neighbourhood.
By the preceding,
$\dot{c}(0)=T_0\Phi(\gamma)=\alpha(\gamma)=v$, whence
$\exp_\cG(v)=c(1)=\exp_G\circ \, \alpha^{-1}(v)$.
\section{Sobolev spaces are suitable for Lie theory}\label{sobolev-suitable}
We now show that the theory
discussed in Sections \ref{sec-genF} to \ref{sec-groups}
applies to Sobolev spaces with real
exponents $s>m/2$.
We start with notation and basic facts.
\begin{numba}
For $m\in\N$ and $s\in[0,\infty[$,
let $H^s(\R^m,\R)$ be the real Hilbert space of equivalence classes
$[\gamma]$
modulo functions vanishing almost everywhere
of $\cL^2$-functions $\gamma\colon \R^m\to\R$
(with respect to Lebesgue-Borel measure $\lambda_m$)
such that $y\mto (1+\|y\|_2^2)^{s/2}\,\wh{\gamma}(y)$
is an $\cL^2$-function as well,
where $\wh{\gamma}$ is the Fourier transform
(see \cite[Appendix~B]{IKT};
cf.\ \cite[Chapter 6.A]{Fol}
as well as Section 1.3.1
and Exercise~1.2.5 in~\cite{Gra},
with $p=2$).
Here $\|\cdot\|_2$ is the Euclidean norm
on~$\R^m$.
The scalar product on $H^s(\R^m,\R)$ is given by
\[
\langle[\gamma],[\eta]\rangle_{H^s}
:=\int_{\R^m}(1+\|y\|_2^2)^{s/2}\, \wh{\gamma}(y)\wb{\wh{\eta}(y)}\,d\lambda_m(y)
\]
for $[\gamma],[\eta]\in H^s(\R^m,\R)$.
We let $\|\cdot\|_{H^s}$ be the corresponding norm,
taking $[\gamma]$ to $\sqrt{\langle [\gamma],[\gamma]\rangle_{H^s}}$.
For $s>m/2$, each $[\gamma]\in H^s(\R^m,\R)$
has a unique
bounded, continuous representative
$\gamma$
and we identify the equivalence
class with this representative.
Moreover, the inclusion map
\[
H^s(\R^m,\R)\to BC(\R^m,\R)
\]
is continuous
(see
Lemma~6.5 in~\cite{Fol}
and the subsequent Remark~1 in~\cite{Fol}).
By definition,
$H^s(\R^m,\R)\sub H^t(\R^m,\R)$
for $s\geq t\geq 0$ and
\begin{equation}\label{norm-order}
\|\gamma\|_{H^t}\leq \|\gamma\|_{H^s}
\quad \mbox{for all $\gamma\in H^s(\R^m,\R)$.}
\end{equation}
\end{numba}
\begin{numba}
If $U\sub\R^m$ is a bounded open set
and $s>m/2$,
we define
\[
H^s(U,\R):=\{\gamma|_U\colon \gamma\in H^s(\R^m,\R)\}
\]
and give this space the quotient norm
with respect to the linear surjection
\[
q_U^s\colon H^s(\R^m,\R)\to H^s(U,\R),\quad \gamma\mto\gamma|_U
\]
whose kernel is closed
as the restriction map $BC(\R^m,\R)\to BC(U,\R)$
is continuous linear (with operator norm $\leq 1$).
%
% eigentlich nicht gebraucht, aber erhellend:
%
The restriction of $q_U^s$ to the orthogonal
complement of $\ker(q^s_U)$ in the Hilbert space $H^s(\R^m,\R)$
is a surjective linear isometry
\[
(\ker(q_U^s))^\perp\to H^s(U,\R),
\]
whose inverse provides an isometric
linear map $\cE^s_U\colon H^s(U,\R)\to H^s(\R^m,\R)$
which is an extension operator:
$\cE^s_U(\gamma)|_U=\gamma$
for all $\gamma\in H^s(U,\R)$.
\end{numba}
\begin{numba}
For a finite-dimensional vector space~$E$,
define $H^s(\R^m,E)$ and $H^s(U,E)$
as in~\ref{topo-funs}.
For $E\cong \R^n$,
the restriction map
$q^{s,E}_U\colon
H^s(\R^m,E)\to H^s(U,E)$, $\gamma\mto\gamma|_U$
then corresponds to $(q^s_U)^n$,
whence it is a quotient map.
For an open subset $V\sub E$,
let $H^s(U,V)$ be the set
of all $\gamma\in H^s(U,E)$
such that $\gamma(U)+Q\sub V$
for some $0$-neighbourhood $Q\sub E$.
Then $H^s(U,V)$ is open
in $H^s(U,E)$, using continuity of
the restriction
map $H^s(U,E)\to BC(U,E)$.
\end{numba}
We shall use the following fact:
\begin{la}\label{nemytskij}
Let $m,n\in \N$, $s\in \;]m/2,\infty[$
and $f\colon \R^m\times \R^n\to \R$ be a bounded
smooth map with bounded partial derivatives
such that
$f(.,0)\in L^2(\R^m,\R)$.
Then $f_*(\gamma):=f\circ(\id_{\R^m},\gamma)\in H^s(\R^m,\R)$
for all $\gamma\in H^s(\R^m,\R^n)$
and the map
$f_*\colon H^s(\R^m,\R^n)\to H^s(\R^m,\R)$
is continuous.
\end{la}
\begin{proof}
It is known that $H^s(\R^m,\R)$
coincides with the Triebel-Lizorkin space
$F^s_{2,2}$ (see (vii) in the proposition
stated on \cite[p.\,14]{RuS}).
Thus, the assertion
follows from Theorem~1
on p.\,387 and Theorem~2 on p.\,389 in~\cite{RuS}.
\end{proof}
\begin{prop}
For $m\in\N$ and $s>m/2$,
the Sobolev spaces $H^s(U,\R)$
on bounded open subsets $U\sub\R^m$
form a family of function spaces which is suitable
for Lie theory.
\end{prop}
\begin{proof}
Axiom~(PF).
Let $U\sub \R^m$ be a bounded open subset
and $V\sub U$ be a relatively compact open subset.
Let $E$ be a finite-dimensional
vector space and $f\colon U\times E\to \R$
be a smooth map. Given $\gamma\in H^s(U,E)$,
we have $\wb{\gamma(U)}\sub W$ for
a relatively compact open set $W\sub E$.
Let $\xi\colon U\to \R$ and $\chi\colon E\to\R$
be compactly supported smooth functions such that
$\xi|_V=1$ and $\chi|_W=1$.
We get a function
$g\in C^\infty_c(\R^m\times E,\R)$ via $g(x,y):=\xi(x)\chi(y)f(x,y)$
for $(x,y)\in U\times E$,
%and
$g(x,y):=0$ for $(x,y)\in (\R^m\setminus \Supp(\xi))\times E$.
Lemma~\ref{nemytskij}
shows that $g_*(\eta):=g\circ (\id_{\R^m},\eta)\in H^s(\R^m,\R)$
for each $\eta\in H^s(\R^m,E)$ and $g_*\colon H^s(\R^m,E)\to H^s(\R^m,\R)$
is continuous. Now
$h(\eta|_U):=g\circ (\id_V,\eta|_V)=g_*(\eta)|_V\in H^s(V,\R)$
is well defined.
Since $h\circ q^{s,E}_U=q^s_U\circ g_*$
is continuous, $h\colon H^s(U,E)\to H^s(V,\R)$ is continuous.
For each $\eta$ in the open $\gamma$-neighbourhood
$H^s(U,W)$, we have $f_*(\eta):=f\circ (\id_V,\eta|_V)=h(\eta)$.
Notably, $f_*(\gamma)\in H^s(V,\R)$ and $f_*\colon H^s(U,E)\to H^s(V,\R)$
is continuous at~$\gamma$.\\[1mm]
Axiom~(PB).
Given an open subset $U\sub\R^m$, let
$H^s_{\loc}(U,\R)$ be the set
of all functions $\gamma\colon U\to\R$
such that $\gamma|_Q\in H^s(Q,\R)$
for each relatively compact
open subset $Q\sub U$.
Let $V$ and $W$ be bounded open subsets of~$\R^m$ such that $\wb{W}\sub U$.
Let $\Theta\colon U\to V$ be a $C^\infty$-diffeomorphism.
If $\gamma\in H^s(V,\R)$,
then $\gamma\in H^s_{\loc}(V,\R)$
in particular, whence
$\gamma\circ\Theta\in H^s_{\loc}(U,\R)$
by \cite[Corollary 6.25]{Fol}
and hence $H^s(\Theta|_W,\R)(\gamma):=\gamma\circ\Theta|_W\in H^s(W,\R)$.
By Remark~\ref{auto-cont}, $H^s(\Theta|_W,\R)\colon H^s(V,\R)\to H^s(W,\R)$
is continuous.\\[1mm]
Axiom (GL).
Let $U$ and $V$ be bounded open subsets in $\R^m$
such that $\wb{V}\sub U$, and $K\sub V$ be compact.
If $\gamma\in H^s(V,\R)$ with $\Supp(\gamma)\sub K$,
there exists $\eta\in H^s(\R^m,\R)$ with $\eta|_U=\gamma$.
Define $\wt{\gamma}(x):=\gamma(x)$ for $x\in V$,
$\wt{\gamma}(x):=0$ for $x\in U\setminus \Supp(\gamma)$.
Let $\xi\in C^\infty_c(\R^m,\R)$
such that $\xi|_K=1$ and $\Supp(\xi)\sub V$.
Then $\xi\eta\in H^s(\R^m,\R)$
by \cite[Proposition~6.12]{Fol}
and $\wt{\gamma}=(\xi\eta)|_U\in H^s(U,\R)$.
By Remark~\ref{auto-cont},
the map $H^s_K(V,\R)\to H^s(U,\R)$, $\gamma\mto\wt{\gamma}$
is continuous.\\[1mm]
Axiom~(MU).
If $U\sub\R^m$ is a bounded open subset
and $h\in C^\infty_c(U,\R)$,
let $\wt{h}\in C^\infty_c(\R^m,\R)$
be the extension of~$h$ by~$0$.
If $\gamma\in H^s(U,\R)$,
let $\eta\in H^s(\R^m,\R)$
with $\eta|_U=\gamma$.
By \cite[Proposition 6.12]{Fol},
$\wt{h}\gamma\in H^s(\R^m,\R)$,
whence $m_h(\gamma):=h\gamma=(\wt{h}\eta)|_U\in H^s(U,\R)$.
By Remark~\ref{auto-cont},
$m_h$ is continuous.
\end{proof}
\section{The Lie groups {\boldmath$H^{>s_0}(M,G)$}}
We begin with preparations for the proof of Theorem~\ref{dirlim-1}.
\begin{la}\label{global-compact}
Let $m\in\N$,
$\cU$ be a good collection
of open subsets of $\R^m$
and $(\cF_j(U,\R))_{U\in\cU}$
be a families of Banach spaces
which are suitable for Lie\linebreak
theory,
for $j\in \{1,2\}$.
Assume that $\cF_1(U,\R)\sub\cF_2(U,\R)$
for each $U\in\cU$,
and the inclusion map
$\cF_1(U,\R)\to\cF_2(U,\R)$
is continuous.
Then we have:
\begin{itemize}
\item[\rm(a)]
$\cF_1(M,N)\sub\cF_2(M,N)$ holds
for each compact $m$-dimensional
smooth manifold~$M$ and finite-dimensional
smooth manifold~$N$.
Moreover,
the inclusion map $\cF_1(M,E)\to\cF_2(M,E)$
is continuous for each finite-dimensional vector
space~$E$.
\item[\rm(b)]
If the mappings
$\kappa_{V,U}\colon
\cF_1(U,\R)\to\cF_2(V,\R)$, $\gamma\mto\gamma|_V$
are compact operators for all $U,V\in\cU$
such that $V$ is relatively compact in~$U$,
then the inclusion mappings
$\kappa^E_M\colon \cF_1(M,E)\to\cF_2(M,E)$
are compact operators
for all finite-dimensional vector spaces~$E$
and all~$M$ as in~{\rm(a)}.
\end{itemize}
\end{la}
\begin{proof}
(a) The hypothesis implies that $(\cF_1)_{\loc}(U,E)\sub(\cF_2)_{\loc}(U,E)$
for each open subset $U\sub\R^m$ and finite-dimensional vector
space~$E$, with continuous linear
inclusion map.
The first assertion follows (using Lemma~\ref{charact_F(M,E)})
and also the second assertion.\\[1mm]
(b)
If $E\cong\R^n$,
then
$\kappa^E_{V,U}\colon \cF_1(U,E)\to\cF_2(V,E)$, $\gamma\mto\gamma|_V$
corresponds to $(\kappa_{V,E})^n$
for all $U,V\in\cU$ with $V$ relatively compact in~$U$,
whence $\kappa^E_{V,E}$ is a compact operator.
There exist $k\in\N$ and charts $\phi_i\colon U_i\to V_i\sub\R^m$
of~$M$ for $i\in\{1,\ldots, k\}$
and sets $W_{i,2}\in \cU$ which are relatively compact in~$V_i$
such that $M=\bigcup_{i=1}^n\phi_i^{-1}(W_{i,2})$.
By Definition~\ref{good-coll}\,(b),
we find $W_{i,1}\in \cU$ such that $W_{i,1}$
is relatively compact in $V_i$
and $\wb{W_{i,2}}\sub W_{i,1}$.
By Proposition~\ref{fin-many}, the map
\[
\Theta_j\colon \cF_j(M,E)\to\prod_{i=1}^k\cF_j(W_{i,j},E),\quad\gamma\mto
(\gamma\circ\phi_i^{-1}|_{W_{i,j}})_{i=1}^k
\]
is a linear topological embedding with
closed image for $j\in\{1,2\}$.
Now $h:=\prod_{i=1}^k\kappa^E_{W_{i,2},W_{i,1}}\colon
\prod_{i=1}^k\cF_1(W_{i,1},E)\to\prod_{i=1}^k\cF_2(W_{i,2},E)$
is a compact operator.
Since $\Theta_2\circ \kappa^E_M= h \circ \Theta_1$
is a compact operator, so is $\kappa^E_M$.
\end{proof}
Following \cite[p.\,200]{Fol},
given $s>m/2$
and a bounded open subset $U\sub\R^m$,
we let $\dot{H}^s(U,\R)$
be the closure of $\{\gamma\in C^\infty_c(\R^m,\R)\colon \Supp(\gamma)\sub U\}$
in $H^s(\R^m,\R)$.
\begin{la}
For each bounded, open subset $U\sub\R^m$
and relatively compact, open subset
$V\sub U$, the continuous linear map
$q^s_{V,U}\colon
\dot{H}^s(U,\R)\to H^s(V,\R)$, $\eta\mto\eta|_V$
is surjective and hence a quotient map.
\end{la}
\begin{proof}
Being a restriction of $q^s_V$, the map $q^s_{V,U}$
is continuous and linear. To see surjectivity,
let $\gamma\in H^s(V,\R)$.
There is $\eta\in H^s(\R^m,\R)$
such that $\eta|_V=\gamma$.
Now the Schwartz space $\cS(\R^m,\R)$
of rapidly decreasing smooth
functions is dense in $H^s(\R^m,\R)$
(cf.\ \cite[p.\,192]{Fol})
and the inclusion mapping\linebreak
$\cS(\R^m,\R)\to H^s(\R^m,\R)$
is continuous, whence $C^\infty_c(\R^m,\R)$
is dense in\linebreak
$H^s(\R^m,\R)$
(being dense in $\cS(\R^m,\R)$ by \cite[Theorem~7.10\,(a)]{Rud}).
Thus, we find a sequence $(\eta_n)_{n\in\N}$
in $C^\infty_c(\R^m,\R)$
such that $\eta_n\to\eta$ in $H^s(\R^m,\R)$
as $n\to\infty$.
Let $\xi\in C^\infty_c(\R^m,\R)$
such that $\Supp(\xi)\sub U$
and $\xi|_V=1$.
Then $\xi\eta_n\to\xi\eta$ in $H^s(\R^m,\R)$
(see \cite[Proposition~6.12]{Fol}),
whence $\xi\eta\in\dot{H}^s(U,\R)$.
Moreover, $(\xi\eta)|_V=\eta|_V=\gamma$.
By the Open Mapping Theorem,
the continuous linear surjection $q^s_{V,U}$ is an open map
and hence a quotient map.
\end{proof}
We recall a known fact.
\begin{la}\label{rettich}
Let $m\in\N$
and
$s>t>m/2$.
For all bounded open subsets $U$
in~$\R^m$,
we have $H^s(U,\R)\sub H^t(U,\R)$.
The inclusion map $\theta_U$ is a compact operator.
\end{la}
\begin{proof}
Let $W$ be a bounded open subset of $\R^m$ with $\wb{U}\sub W$.
By Rellich's Theorem \cite[Theorem 6.14]{Fol},
the inclusion map
$h\colon \dot{H}^s(W,\R)\to\dot{H}^t(W,\R)$
is a compact operator.
Then $\theta_U\circ q^s_{U,W}=q^t_{U,W}\circ h$
is a compact operator
and hence continuous, whence $\theta_U$
is continuous. If $B\sub \dot{H}^s(W,\R)$
is a bounded open $0$-neighbourhood,
then $q^s_{U,W}(B)$ is a bounded $0$-neighbourhood in $H^s(U,\R)$.
As $\theta_U(q^s_{U,W}(B))$
is relatively compact,
$\theta_U$
is a compact operator.
\end{proof}
\begin{numba}
A locally convex space $E$ is called a \emph{Silva space} (or (DFS)-space)
if $E$ is the locally convex direct limit
of some Banach spaces $E_1\sub E_2\sub\cdots$,
such that all inclusion maps
$E_j\to E_{j+1}$
are compact operators
(see, e.g., \cite{Flo} or \cite[Appendix~B.13]{GaN}).
Every Silva space is complete.
It is \emph{compact regular} in the sense
that each compact subset $K\sub E$ is a compact subset
of some~$E_j$.
The locally convex topology~$\cO$ on~$E$
then also makes~$E$ the direct limit of the $E_n$
as a topological space (see the cited works).
Thus, a subset $U\sub E$ is open if and only if
$U\cap E_j$ is open in~$E_j$ for each $j\in\N$.
\end{numba}
\begin{la}\label{superpo-dl}
Let $m\in\N$,
$\cU$ be a good collection of open
subsets of~$\R^m$ and $(\cF_j(U,\R))_{U\in\cU}$
be a family of Banach spaces which is suitable
for Lie theory, for each $j\in\N$.
For all $j\in\N$ and $U\in\cU$,
assume that $\cF_j(U,\R)\sub \cF_{j+1}(U,\R)$
with continuous inclusion map.
For all $j\in \N$ and $U,V\in\cU$
with $V$ relatively compact in~$U$,
assume that the map
$\cF_j(U,\R)\to\cF_{j+1}(V,\R)$,
$\gamma\mto \gamma|_V$ is a compact operator.
Then the following holds:
\begin{itemize}
\item[\rm(a)]
For each finite-dimensional vector space~$E$,
the locally convex direct limit topology
makes $\cF(M,E):=\bigcup_{j\in\N}\cF_j(M,E)$
a Silva space.
\item[\rm(b)]
If $E$ and $F$ are finite-dimensional
vector spaces, $U\sub E$ is open and
$\Phi\colon U\to F$ a $C^\infty$-map,
then
$\cF(M,U):=\{\gamma\in\cF(M,E)\colon \gamma(M)\sub U\}$
is an open subset of $\cF(M,E)$
and the map
\[
\cF(M,\Phi)\colon \cF(M,U)\to\cF(M,F),\quad
\gamma\mto\Phi\circ\gamma
\]
is smooth.
If $E$ and $F$ are $\K$-vector spaces for $\K\in\{\R,\C\}$
and $\Phi$ is $\K$-analytic, then also $\cF(M,\Phi)$
is $\K$-analytic.
\end{itemize}
\end{la}
\begin{proof}
(a) By Lemma~\ref{global-compact},
$\cF_j(M,E)\sub\cF_{j+1}(M,E)$
and the inclusion map is a compact operator.\\[1mm]
(b) $\cF(M,U)$ is open in the Silva space $\cF(M,E)$
as $\cF(M,U)\cap \cF_j(M,E)=\cF_j(M,U)$ is open in $\cF_j(M,E)$
for each $j\in\N$. The inclusion mapping\linebreak
$\Lambda_j^F\colon \cF_j(M,F)\to\cF(M,F)$
is continuous and linear. Since $\cF(M,\Phi)|_{\cF_j(M,U)}=\Lambda^F_j\circ\cF_j(M,\Phi)$
is smooth for each $j\in \N$ by Corollary~\ref{left_comp_smooth_Manifold},
also $\cF(M,\Phi)$
is smooth (see \cite[Lemma~9.7]{COM}).
The complex analytic case follows
in the same way, using Corollary~\ref{superpo-cx}.
If $\Phi$ is real analytic, pick
a complex analytic extension $\Psi\colon V\to F_\C$
of $\Phi$, defined on an open subset $V\sub E_\C$.
Then $\cF(M,\Psi)$ is a complex
analytic extension for $\cF(M,\Phi)$.
\end{proof}
Before we can prove Theorem~\ref{dirlim-1},
we recall further terminology.
\begin{numba}\label{defn-reg-gp}
Let $G$ be a Lie group with neutral element~$e$
and Lie algebra $\cg:=T_eG$.
Let $G\times TG\to TG$, $(g,v)\mto g.v$
be the left action of $G$ on its tangent bundle
given by $g.v:=TL_g(v)$,
where $L_g\colon G\to G$, $x\mto gx$.
Given $k\in\N_0\cup\{\infty\}$,
endow $C^k([0,1],\cg)$
with the topology of uniform convergence
of $C^k$-functions $\gamma\colon [0,1]\to\cg$
and their derivatives up to $k$th order.
The Lie group~$G$ is called
\emph{$C^k$-regular} if, for each $\gamma\in C^k([0,1],\cg)$,
the initial value problem
\begin{equation}\label{inival}
\dot{\eta}(t)=\eta(t).\gamma(t),\quad \eta(0)=e
\end{equation}
has a (necessarily unique) solution
$\eta\colon [0,1]\to G$ and the evolution map
$C^k([0,1],\cg)\to G$,
$\gamma\mto \eta(1)$
is smooth (see \cite{SEM}).
Every $C^k$-regular Lie group is $C^\infty$-regular,
a concept going back to~\cite{Mil}
(for sequentially complete~$\cg$).\\[2.3mm]
We shall also encounter
$L^\infty_{\rc}$-regularity of Lie groups modelled
on sequentially complete locally convex spaces,
a more specialized property introduced in~\cite{MEA}
(see 1.13 and Definition~5.16 in loc.\,cit.)
We shall not repeat the concept here
but recall that $L^\infty_{\rc}$-regularity
implies $C^0$-regularity (cf.\ \cite[Corollary 5.21]{MEA}).
\end{numba}
{\bf Proof of Theorem~\ref{dirlim-1}.}
\emph{The modelling space.}
We pick $s_1>s_2>\cdots$ in $\,]s_0,\infty[$
such that $s_j\to s_0$ as $j\to\infty$.
For each $j\in \N$, we have $H^{s_j}(M,\cg)\sub H^{s_{j+1}}(M,\cg)$
and the inclusion map is a compact operator,
as a consequence of Lemmas~\ref{global-compact}
and \ref{rettich}.
Thus, the locally convex direct limit topology makes
\[
H^{>s_0}(M,\cg)=\bigcup_{j\in\N}H^{s_j}(M,\cg)=\dl\,H^{s_j}(M,\cg)
\]
a Silva space. Note that $\,]s_0,\infty[$
is a directed set for the opposite of the
usual order. As $(s_j)_{j\in\N}$
is a cofinal subsequence of the latter set,
we have
\[
\dl_{s>s_0}H^s(M,\cg)=\dl\, H^{s_j}(M,\cg)
\]
in a standard way. The same argument allows
$]s_0,\infty[$ to be replaced with $\{s_j\colon j\in\N\}$
in the direct limit properties described in
Theorem~\ref{dirlim-1}.\\[2.3mm]
\emph{The group.}
By Lemma~\ref{global-compact}\,(a), $H^{s_j}(M,G)$ is a subgroup
of $H^{s_{j+1}}(M,G)$
for each $j\in\N$.
We give
$H^{>s_0}(M,G)=\bigcup_{j\in\N}H^{s_j}(M,G)$
the unique group structure
making $H^{s_j}(M,G)$ a subgroup
for each $j\in\N$.\\[2.3mm]
\emph{The map $H^{>s_0}(M,\exp_G)$.}
For each $\gamma\in H^{>s_0}(M,\cg)$,
we have $\gamma\in H^{s_j}(M,\cg)$
for some $j\in \N$ and hence
$H^{>s_0}(M,\exp_G)(\gamma):=\exp_G\circ \, \gamma= H^{s_j}(M,\exp_G)(\gamma)
\in H^{s_j}(M,G)\sub H^{>s_0}(M,G)$,
using Lemma~\ref{PF-between}.\\[2.3mm]
\emph{The adjoint action on $H^{>s_0}(M,\cg)$.}
If $\gamma\in H^{>s_0}(M,G)$,
then $\gamma\in  H^{s_j}(M,G)$ for some $j\in\N$.
For all $i\geq j$, we then have
$\gamma\in H^{s_i}(M,G)$
and the proof of Proposition~\ref{fi-prop} shows that
\[
\beta_i\colon H^{s_i}(M,\cg)\to H^{s_i}(M,\cg),\quad
\eta\mto \Ad\circ(\gamma,\eta)
\]
is a continuous linear map (where
$\Ad\colon G\times\cg\to\cg$ is the adjoint action).
Then also the linear map
\begin{equation}\label{new-Ad}
\beta\colon H^{>s_0}(M,\cg)\to H^{>s_0}(M,\cg),\quad\eta\mto \Ad\circ(\gamma,\eta)
\end{equation}
is continuous, as $\beta=\dl_{i\geq j}\beta_i$.\\[2.3mm]
\emph{The Lie group structure.}
We already saw that $H^{>s_0}(M,G)$
is a group under pointwise operations and that
$\exp_G\circ\,\gamma\in H^{>s_0}(M,G)$
for all $\gamma\in H^{>s_0}(M,\cg)$.
To construct
the Lie group structure on $H^{>s_0}(M,G)$,
replace $\cF$ with $H^{>s_0}$
in the remaining steps of the proof of Proposition~\ref{fi-prop}
and make the following changes:
We use Lemma~\ref{superpo-dl}
in place of Corollary~\ref{superpo-cx};
we use the continuity
of~$\beta$ in~(\ref{new-Ad})
just established
instead of Lemma~\ref{for-Ad}.\\[2.3mm]
As a result, $\cG:=H^{>s_0}(M,G)$
is a $\K$-analytic BCH-Lie group
modelled on $H^{>s_0}(M,\cg)$.
For each $x\in M$,
the point evaluation $\ev_x\colon\cG\to G$
is a $\K$-analytic group homomorphism and
\[
\alpha^{-1}\colon L(\cG)\to H^{>s_0}(M,\cg),\quad
v\mto (L(\ev_x)(v))_{x\in M}
\]
is an isomorphism of topological Lie algebras
if we endow $H^{>s_0}(M,\cg)$ with the pointwise
Lie bracket. Moreover,
$H^{>s_0}(M,\exp_G)\circ \alpha^{-1}$
is the exponential function of~$\cG$.\\[2.3mm]
\emph{Existence of a direct limit chart.}
With $P$, $Q$, $U$, $V$, $\phi$
as in the preceding adaptation of the proof of Proposition~\ref{fi-prop},
the map
\[
\Phi:=H^{>s_0}(M,\phi|_V)\colon H^{>s_0}(M,V)
\to H^{>s_0}(M,U)
\]
is a $\K$-analytic diffeomorphism and
$\Phi^{-1}$ is a chart for $H^{>s_0}(M,G)$
whose restriction to $H^{>s_0}(M,U)\cap H^{s_j}(M,G)=H^{s_j}(M,U)$
is the chart
\[
H^{s_j}(M,\phi^{-1}|_U)\colon H^{s_j}(M,U)\to H^{s_j}(M,V),\quad \gamma\mto \phi^{-1}|_U
\circ
\gamma
\]
of the Lie group $H^{s_j}(M,G)$ around~$e$.
Thus $\Phi^{-1}$
is a strict direct limit chart for $H^{>s_0}(M,G)=\bigcup_{j\in\N}H^{s_j}(M,G)$
around~$e$ as in \cite[Definition~2.1]{COM}.\\[2.3mm]
\emph{Regularity.}
Since $H^{>s_0}(M,\cg)=\dl\, H^{s_j}(M,\cg)$
is a Silva space and thus compact regular,
the Lie group $H^{>s_0}(M,G)=\bigcup_{j\in\N}H^{s_j}(M,G)$
is $L^\infty_{\rc}$-regular by
\cite[Proposition~8.10]{MEA}
and hence $C^0$-regular.\\[2.3mm]
\emph{Direct limit properties.}
Since $\cG:=H^{>s_0}(M,G)=\bigcup_{j\in\N}H^{s_j}(M,G)$
has a direct limit chart and
$L(\cG)\cong H^{>s_0}(M,\cg)=\dl\, H^{s_j}(M,\cg)$
is a Silva space,
\cite[Proposition 9.8\,(i)]{COM}
shows that $\cG=\dl\,H^{s_j}(M,G)$
as a topological group, $C^\infty_\bL$-Lie group for $\bL\in\{\R,\K\}$,
and as a $C^r_\bL$-manifold for all $r\in\N_0\cup\{\infty\}$.\\[2.3mm]
\emph{Compact subsets.}
Since $\cG$ has a direct limit chart
and $L(\cG)\colon H^{>s_0}(M,\cg)=\bigcup_{j\in\N}H^{s_j}(M,\cg)$
is compact regular, each compact subset $K$
of~$\cG$ is a compact subset of $H^{s_j}(M,G)$
for some $j\in\N$,
by \cite[Lemma~6.1]{HGS}.
\appendix
\section{Bounded open sets with smooth boundary}\label{appA}
Let $m\in\N$. We show that the set $\cU$ of all bounded, open subsets
$U\sub\R^m$ with smooth boundary
is a good collection of open subsets of~$\R^m$.
\begin{defn}\label{smoo-bd}
A compact subset $L\sub\R^m$
is called a \emph{compact subset with smooth boundary}
if, for each $x\in \partial L$,
there exists a $C^\infty$-function $g\colon Q\to\R$
on an open $x$-neighbourhood $Q\sub\R^m$
such that $\nabla g(y)\not=0$
for all $y\in Q$ and
\[
L\cap Q=\{y\in Q\colon g(y)\leq 0\}.
\]
We say that a bounded open subset $U\sub \R^m$
has smooth boundary
if $\wb{U}$ is a compact subset of~$\R^m$
with smooth boundary and $U=\wb{U}^{\,0}$.
\end{defn}
\begin{rem}\label{aequi}
For $x\in\partial L$
and $g$ as in Definition~\ref{smoo-bd},
after a permutation of the coordinates we
may assume that $\frac{\partial g}{\partial x_n}(x)\not=0$.
After shrinking $Q$, we may
assume that 
$\frac{\partial g}{\partial x_n}(y)> 0$
for all $y\in Q$ (which we assume now)
or $\frac{\partial g}{\partial x_n}(y)<0$
for all $y\in Q$ (an analogous case).
Shrinking $Q$ further,
we may assume that $Q=W\times J$
for an open set $W\sub\R^{m-1}$ and an open
interval $J\sub\R$
and that
\[
\{y\in Q\colon g(y)=0\}=\graph(h)
\]
for a smooth function $h\colon W \to J$,
by the Implicit Function Theorem.
Then
\[
Q\cap L=\{(w,t)\in W\times J\colon t\leq h(w)\}
\]
by monotonicity of $g(w,\cdot)$ on~$J$.
Notably, $\{(w,t)\in W\times J\colon t<h(w)\}\sub L^0$
is dense in $Q\cap L$, whence
$L^0$ is dense in~$L$. Moreover, $Q\cap\partial L=\graph(h)$.
\end{rem}
It is easy to see that $\cU$
satisfies the conditions~(a) and~(d)
formulated in Definition~\ref{good-coll}.
To see that~(b) holds,
let $U\sub\R^m$ be a bounded open
subset with smooth boundary
and $K\sub U$ be a non-empty
compact subset. Thus $L:=\wb{U}$
is a compact subset of~$\R^m$ with smooth boundary
and $U=L^0$.
Then $\partial L$
is a compact smooth submanifold of $\R^m$
and we consider the inner normal vector field
\[
\nu\colon \partial L\to \R^m
\]
given for $y\in Q\cap \partial L$ (with $Q$ as in Definition~\ref{smoo-bd}) by
\[
\nu(y)=-\frac{1}{\|\nabla g(y)\|_2}\nabla g(y).
\]
Thus $\nu(y)$ is the unique unit vector in $(T_y(\partial L))^\perp$
such that $y+t\nu(y)\in L$ for all small $t\geq 0$.
The hypotheses~(d) of \cite[Theorem~1.10]{SMO}
being satisfied, its conclusion~(i)
provides a smooth vector field
$F\colon \R^m\to\R^m$ with $F|_{\partial L}=\nu$.
Using a smooth partition of unity,
we can create a compactly supported
smooth function
$\xi\colon\R^m\to\R$
such that $\xi|_{\partial L}=1$
and $\Supp(\xi)\sub \R^m\setminus K$.
After replacing $F$ with $\xi F$,
we may assume that $F$ has compact support
and $\Supp(F)\cap K=\emptyset$.
For each $y\in \R^m$, the maximal solution $\phi_y$
of the initial value problem
\[
x'(t)=F(x(t)),\quad x(0)=y
\]
is defined for all $t\in\R$.
We now use a standard fact
concerning flows of complete vector fields:
Setting $\Fl_t(y):=\phi_y(t)$ for $y\in\R^m$,
we get $C^\infty$-diffeomorphisms
$\Fl_t\colon\R^m\to\R^m$
for all $t\in\R$. 
If $x\in\partial L$,
let $g\colon Q\to\R$ be as in Definition~\ref{smoo-bd}.
There is $\ve>0$ such that
$\phi_x(]{-\ve},\ve[)\sub Q$.
Since
\[
(g\circ \phi_x)'(0)=\langle \nabla g(\phi_x(0)),
\phi_x'(0)\rangle=\langle \nabla g(x),\nu(x)\rangle
=-\|\nabla g(x)\|_2<0,
\]
after shrinking $\ve$, we can achieve that
$(g\circ\phi_x)'(t)<0$ for all $t\in\;]{-\ve},\ve[$.
Thus $g(\phi_x(t))<0$ (and hence $\phi_x(t)\in L^0=U$)
for all $t\in \;]0,\ve[$,
while $g(\phi_x(t))>0$ (and hence $\phi_x(t)\in\R^m\setminus L$)
for all $t\in \;]{-\ve},0[$.\\[2.3mm]
We now show that, for each $y\in L$,
we have
\begin{equation}\label{get-inside}
\phi_y(t)\in L^0\quad\mbox{for all $t>0$.}
\end{equation}
If this was wrong, we could define
\begin{equation}\label{is-inf}
\tau:=\inf\{t>0\colon \phi_y(t)\not\in L^0\}.
\end{equation}
Then $\tau>0$, as we just observed that
$\phi_y(t)\in L^0$ for small $t>0$
if $y\in \partial L$;
the corresponding statement for
$y\in L^0$ also holds as $\phi_y^{-1}(L^0)$
is an open $0$-neighbourhood in this case.
Since $\phi_y$ is continuous
and $\R^m\setminus L^0$ is closed,
we have $\phi_y(\tau)\in\R^m\setminus L^0$.
On the other hand, $\phi_y(t)\to \phi_y(\tau)$
as $[0,\tau[\;\ni t\to \tau$,
whence $\phi_y(\tau)\in L$ and hence
$x:=\phi_y(\tau)\in L\setminus L^0=\partial L$.
But then $\phi_y(\tau-t)=\phi_x(-t)\in \R^m\setminus L$
for all small $t>0$, contradicting~(\ref{is-inf}).\\[2.3mm]
Fix a real number $t_0>0$. Then $\Fl_{t_0}\colon\R^m\to\R^m$
is a $C^\infty$-diffeomorphism,
whence $\Fl_{t_0}(L)$ is a compact subset of~$\R^m$
with smooth boundary and $V:=\Fl_{t_0}(U)=(\Fl_{t_0}(L))^0$
a bounded open subset of~$\R^m$
with smooth boundary.
By (\ref{get-inside}),
we have $\wb{V}=\Fl_{t_0}(L)\sub L^0=U$.
Note that $P:=L^0\setminus \Supp(F)$
is an open subset of~$\R^m$ such that $K\sub P$.
Since $F(x)=0$ for all $x\in P$,
we have $\phi_x(t)=x$ for all $t\in\R$
and hence $\Fl_{t_0}(x)=x$.
Thus $K\sub P=\Fl_{t_0}(P)\sub \Fl_{t_0}(U)=V$.\\[1mm]
To get~(c), let $O$ be an open subset of $\R^m$
and $U\not=\emptyset$ be a relatively compact subset of~$O$
such that $U\in\cU$.
We construct a relatively compact subset $W$ of~$O$
such that $\wb{U}\sub W$ and $W\in \cU$.
Let
$F$ and $\Fl_t$
be as in the proof of~(b),
applied with a singleton $K\sub U$.
It is a standard fact that the map
\[
\R\times \R^m\to\R^m,\quad (t,y)\mto\Fl_t(y)
\]
is smooth and hence continuous.
Thus $S:=\{(t,y)\in\R\times \R^m\colon \Fl_t(y)\in O\}$
is open in $\R\times \R^m$.
Since $\Fl_0=\id_{\R^m}$,
we have $\{0\}\times \wb{U}\sub S$.
Using the Wallace Theorem (see Theorem~12 in \cite[Chapter~5]{Kel}),
we find an open $0$-neighbourhood $J\sub \R$
and an open subset $Y\sub\R^m$ with $\wb{U}\sub Y$ such that
$J\times Y\sub S$. We pick $t_0\in J$ such that $t_0<0$.
Then $W:=\Fl_{t_0}(U)$ is a bounded open subset of~$\R^m$
with smooth boundary.
Since $U\supseteq \Fl_{-t_0}(\wb{U})$,
%we have
\[
\wb{U}=\Fl_{t_0}(\Fl_{-t_0}(\wb{U}))\subseteq \Fl_{t_0}(U)=W
\]
follows. Moreover,
$\wb{W}=\Fl_{t_0}(\wb{U})\sub \Fl_{t_0}(Y)\sub O$. $\square$\\[2.3mm]
The preceding proof varies the discussion of flows
of inner vector fields on manifolds with corners
in \cite[\S2.7]{Mic}.
{\bf Helge  Gl\"{o}ckner}, Institut f\"{u}r Mathematik, Universit\"at Paderborn,\\
Warburger Str.\ 100, 33098 Paderborn, Germany; {\tt  glockner@math.upb.de}\\[2.5mm]
{\bf Luis T\'{a}rrega},
Universitat Jaume I, Departamento de Matem\'{a}ticas,
Campus de Riu Sec, 12071 Castell\'{o}n, Spain;
{\tt ltarrega@uji.es}\vfill

\begin{thebibliography}{99}
%
%
\bibitem{Bas}
Bastiani, A.,
\emph{Applications diff\'{e}rentiables et vari\'{e}t\'{e}s diff\'{e}rentiables de dimension infinie},
J. Anal.\ Math.\ {\bf 13} (1964), 1--114.
%
%
\bibitem{BGN}
Bertram, W., H. Gl\"{o}ckner, and K.-H. Neeb,
\emph{Differential calculus over general base fields and rings},
Expo.\ Math.\ {\bf 22} (2004), 213--282.
%
%
\bibitem{BaS}
Bochnak, J. and J. Siciak,
\emph{Analytic functions in topological vector spaces},
Stud.\ Math.\ {\bf 39} (1971), 77--112.
%
%
\bibitem{Bou}
Bourbaki, N.,
``Lie Groups and Lie Algebras, Chapters 1--3,''
Springer,
Berlin, 1989.
%
%
\bibitem{Eel}
Eells, J. Jr.,
\emph{A setting for global analysis},
Bull.\ Amer.\ Math.\ Soc.\ {\bf 72} (1966), 751--807.
%
%
\bibitem{Flo}
Floret, K.,
\emph{Lokalkonvexe Sequenzen mit kompakten Abbildungen},
J. Reine Angew.\ Math.\ {\bf 247} (1971), 155--195.
%
%
\bibitem{Fol}
Folland, G.\,B.,
``Introduction to Partial Differential Equations,''
2nd Edition, Princeton University Press,
Princeton, 1995.
%
%
\bibitem{Fre}
Freed, D.\,S.,
\emph{The geometry of loop groups},
J. Differential Geom.\ {\bf 28} (1988),
%no. 2,
223--276.
%
%
\bibitem{FaU}
Freed, D.\,S.
and K.\,K. Uhlenbeck, ``Instantons and Four-Manifolds,''
Springer, New York, 1984.
%
%
\bibitem{RES}
Gl\"{o}ckner, H.,
\emph{Infinite-dimensional Lie groups without completeness restrictions},
pp.\ 43--59 in:
A. Strasburger, J. Hilgert, K.-H. Neeb,
and W. Woyjty\'{n}ski (eds.),
``Geometry and Analysis on Finite- and Infinite-Dimensional Lie Groups,''
Banach Center Publ.\ {\bf 55}, Warsaw, 2002.
%
%
\bibitem{GCX}
Gl\"{o}ckner, H.,
\emph{Lie group structures on quotient groups and universal complexifications for infinite-dimensional Lie groups},
J. Funct.\ Anal.\ {\bf 194} (2002), 347--409.
%
%
\bibitem{DIR}
Gl\"{o}ckner, H.,
\emph{Direct limit Lie groups and manifolds},
J. Math.\ Kyoto Univ.\ {\bf 43} (2003),
%no. 1,
2--26.
%
%
\bibitem{COM}
Gl\"{o}ckner, H.,
\emph{Direct limits of infinite-dimensional Lie groups compared to direct limits in related categories},
J. Funct.\ Anal.\ {\bf 245} (2007), 19--61. 
%
%
\bibitem{HGS}
Gl\"{o}ckner, H.,
\emph{Direct limits of infinite-dimensional Lie groups},
pp.\ 243--280 in: K.-H. Neeb and A. Pianzola (eds.), ``Developments and Trends in Infinite-Dimensional Lie Theory,'' Birkh\"{a}user, Basel,
2011.
%
%
\bibitem{SEM}
Gl\"{o}ckner, H.,
\emph{Regularity properties of infinite-dimensional Lie groups,
and semiregularity},
preprint, arXiv:1208.0715.
%
%
\bibitem{MEA}
Gl\"{o}ckner, H.,
\emph{Measurable regularity properties of infinite-dimensional Lie groups},
preprint, arXiv:1601.02568.
%
%
\bibitem{SMO}
Gl\"{o}ckner, H.,
\emph{Smoothing operators for vector-valued functions and extension operators},
preprint, arXiv:2006.00254.
%
%
\bibitem{FUR}
Gl\"{o}ckner, H.,
\emph{Direct limits of mapping groups},
in preparation.
%
%
\bibitem{GaN}
Gl\"{o}ckner, H. and K.-H. Neeb,
``Infinite Dimensional Lie Groups,''
book in preparation.
%
%
\bibitem{Gra}
Grafakos, L.,
``Modern Fourier Analysis,''
3rd Edition, Springer, New York, 2014.
%
%
\bibitem{Ham}
Hamilton, R.\,S.,
\emph{The inverse function theorem of Nash and Moser},
Bull.\ Amer.\ Math.\ Soc.\ {\bf 7} (1982),
%no. 1,
65--222. 
%
%
\bibitem{HaM}
Hekmati, P. and J. Mickelsson,
\emph{Fractional loop group and twisted $K$-theory},
Commun.\ Math.\ Phys.\ {\bf 299} (2010),
%No. 3,
741--763.
%
%
\bibitem{IKT}
Inci, H., T. Kappeler, and P. Topalov,
``On the Regularity of the Composition of Diffeomorphisms,''
Mem.\ Amer.\ Math.\ Soc.\ {\bf 226} (2013), no.\ 1062.
%
%
\bibitem{Kel}
Kelley, J.\,L.,
``General Topology,''
Springer, New York,
1975.
%
%
\bibitem{KaM} Kriegl, A. and P.\,W. Michor,
``The Convenient Setting of Global Analysis,''
AMS, Providence, 1997.
%
%
\bibitem{Mic} Michor, P.\,W., ``Manifolds of Differentiable Mappings,''
Shiva Publ., Orpington, 1980.
%
%
\bibitem{Mck}
Mickelsson, J.,
``Current Algebras and Groups,''
Plenum Press, New York, 1989.
%
%
\bibitem{Mil} Milnor, J., \emph{Remarks on infinite-dimensional Lie groups},
pp.\,1007--1057 in: B.\,S. DeWitt and R. Stora (eds.),
``Relativit\'{e}, groupes et topologie~II,'' North-Holland,
Amsterdam, 1984.
%
%
\bibitem{Nee}
Neeb, K.-H.,
\emph{Towards a Lie theory of locally convex groups},
Jpn.\ J. Math.\ {\bf 1} (2006), 291--468.
%
%
\bibitem{Om2}
Omori, H.,
``Infinite-Dimensional Lie Groups,''
Amer.\ Math.\ Soc., 1997.
%
% 
\bibitem{Pal}
Palais, R.\,S.,
``Foundations of Global Non-Linear Analysis,''
W.\,A. Benjamin, New York, 1968.
%
%
\bibitem{Pic}
Pickrell, D.,
\emph{Heat kernel measures at critical limits},
pp.\ 393--415 in: K.-H. Neeb and A. Pianzola (eds.), ``Developments and Trends in Infinite-Dimensional Lie Theory,'' Birkh\"{a}user, Basel,
2011.
%
%
\bibitem{PaS}
Pressley, A. and G. Segal, ``Loop Groups,''
Oxford University Press, New York, 1986.
%
%
\bibitem{Rud}
Rudin, W., ``Functional Analysis,''
McGraw Hill, 1991.
%
%
\bibitem{RuS}
Runst, T. and W. Sickel,
``Sobolev Spaces of Fractional Order, Nemytskij Operators,
and Nonlinear Partial Differential Equations,''
De Gruyter, Berlin, 1996. 
%
%
\bibitem{Wei}
von Weizs\"{a}cker, H.,
\emph{In which spaces is every curve Lebesgue-Pettis integrable}?
preprint, arXiv:1207.6034.\vspace{1mm}
%
\end{thebibliography}
\end{document}